\documentclass[a4paper,12pt,oneside,reqno]{amsart}
\usepackage{amssymb,amsfonts,amsmath,amsthm,amscd, bbm}
\usepackage{graphicx, color}
\usepackage{bm}
\usepackage{enumerate,enumitem}
\usepackage[%
ocgcolorlinks,%
linkcolor=black,
filecolor=black,%
citecolor=black,%
urlcolor=black]{hyperref}
\numberwithin{equation}{section}  

\newtheorem{teor}{Theorem}[section]
 
\newtheorem{lem}[teor]{Lemma}

\newtheorem{rem}[teor]{Remark}

\newcommand{\beq}{\begin{equation}}
\newcommand{\eeq}{\end{equation}}
\newcommand{\bea}{\begin{aligned}}
\newcommand{\eea}{\end{aligned}}

\newcommand{\R}{\mathbb R}
\newcommand{\N}{\mathbb N}
\newcommand{\E}{\mathbb E}
\newcommand{\C}{{\mathbb C}}
\newcommand{\bbP}{\mathbb P}
\newcommand{\tr}{{\rm tr} \,}
\newcommand{\rme}{\mathrm e}

\newcommand{\EE}{{\sf{E}}}

\newcommand{\bfg}{{\boldsymbol g}}
\newcommand{\bfm}{{\boldsymbol m}}
\newcommand{\bfh}{{\boldsymbol h}}
\newcommand{\TAP}{{\rm TAP}}
\newcommand{\supp}{{\rm supp\,}}
\newcommand{\sign}{{\rm sign\,}}
\newcommand{\diag}{{\rm diag\,}}
\newcommand{\rmd}{{\rm d}}

\newcommand{\wt}{\widetilde}
\newcommand{\wh}{\widehat}
\newcommand{\bs}{\boldsymbol}

\newcommand{\ch}{}
\newcommand{\bch}{}
\newcommand{\ech}{\normalcolor}

\begin{document}

\title{On the concavity of the TAP free energy in the SK model}

\author[S. Gufler]{Stephan Gufler}
\address{S. Gufler  \\ J.W. Goethe-Universit\"at Frankfurt, Germany.}
\email{gufler@math.uni-frankfurt.de}

\author[A. Schertzer]{Adrien Schertzer}
\address{A. Schertzer \\  Institut f\"ur Angewandte Mathematik, Bonn University, Germany }
\email{aschertz@uni-bonn.de}

\author[M. A. Schmidt]{Marius A. Schmidt}
\address{M. A. Schmidt \\ J.W. Goethe-Universit\"at Frankfurt, Germany.}
\email{m.schmidt@mathematik.uni-frankfurt.de}

\thanks{We are grateful to Nicola Kistler for suggesting the problem of concavity of the TAP free energy and for helpful discussions. We also thank Christian Brennecke and V\'eronique Gayrard  for valuable discussions, and Jan Lukas Igelbrink for help with numerical simulations. We are also grateful to Mireille Capitaine for explaining to us how~\cite{CDFF} can be extended to the GOE case, and for telling us about the reference~\cite{Fan}. \ch We would like to thank an anonymous referee for helpful comments and pointers to the literature. \ech
	This work was partly funded by the Deutsche Forschungsgemeinschaft (DFG, German Research Foundation) under Germany’s Excellence Strategy - GZ 2047/1, Projekt-ID 390685813 and GZ 2151 - Project-ID 390873048, through Project-ID 211504053 - SFB 1060, and by DFG research grant contract number 2337/1-1, project 432176920.
}


\date{\today}

\begin{abstract}
We analyse the Hessian of the Thouless-Anderson-Palmer (TAP) free energy for the Sherrington-Kirkpatrick model, below the de Almeida-Thouless line, evaluated in Bolthausen's approximate solutions of the TAP equations. \ch We show that the empirical spectral distribution weakly converges to a measure with negative support below the AT line, and that the support includes zero on the AT line. In this ``macroscopic'' sense, we show that TAP free energy is concave in the order parameter of the theory, i.e.\ the random spin-magnetisations. This proves a spectral interpretation of the AT line. We also find different magnetizations than Bolthausen's approximate solutions at which the Hessian of the TAP free energy has positive outlier eigenvalues. In particular, when the magnetizations are assumed to be independent of the disorder, we prove that Plefka's second condition is equivalent to all eigenvalues being negative. On this occasion, we extend the convergence result of Capitaine et al.\ (Electron. J. Probab. {\bf 16}, no. 64, 2011) for the largest eigenvalue of perturbed complex Wigner matrices to the GOE.\ech
\end{abstract}

\maketitle
\setcounter{tocdepth}{1}
\tableofcontents

\section{Introduction}
We consider the standard Sherrington-Kirkpatrick (SK for short) model with an external field. In its random Hamiltonian
\begin{equation}\label{e:def-Hamiltonian}
H_{\beta,h}(\sigma):=\frac{\beta}{\sqrt{2N}}\sum_{i,j=1}^{N} g_{ij}\sigma_i \sigma_j+h\sum_{i=1}^{N}\sigma_i
\end{equation}
for $N\in\N$ spins $\sigma=(\sigma_i)\in \Sigma_N:=\{-1,1\}^N$, the disorder is modeled by i.i.d.\ centered Gaussians $g_{ij}$ with variance $1$ on a probability space $(\Omega, \mathcal{F} , \bbP )$. The parameters $\beta>0$ and $h \in \R$ are called inverse temperature and external field.
The partition function is given by
\begin{equation}
Z_N(\beta,h) := 2^{-N} \sum_{\sigma} \exp {H_{\beta,h}(\sigma)},
\end{equation}
and the free energy by
\begin{equation}
f_N(\beta,h) := \frac{1}{N} \log Z_N(\beta,h).
\end{equation}
A well-known consequence of Gaussian concentration of measure is that the free energy is self-averaging in the sense that
\begin{equation}
f(\beta,h) := \lim_{N \to \infty}\frac{1}{N} \log Z_N(\beta,h) = \lim_{N \to \infty}\frac{1}{N} \E\log Z_N(\beta,h) \text{ almost surely.}
\end{equation}
The existence of the limit on the right-hand side was established in a celebrated paper by Guerra and Toninelli \cite{Guerra}. The limit is given by the Parisi variational formula (see \cite{t,panchenko,Guerra1}). In high temperature ($\beta$ small), $f(\beta,h)$ is also given by the replica-symmetric formula, originally proposed by Sherrington and Kirkpatrick \cite{sk}:
\begin{teor}[\cite{sk,b2,chr}]\label{t:RS}
	There exists $\beta_0> 0$ such that for all $h, \, \beta$ with $0<\beta \leq \beta_0,$
	 \begin{equation}\label{e:t:RS}
		 f(\beta,h)=RS(\beta,h):=\inf_{q\geq0}\big\{\EE\log\cosh(h+\beta\sqrt{q}Z) +\frac{\beta^2(1-q)^2}{4}\big\},
	\end{equation}
	where $Z$ is a standard Gaussian.
\end{teor}
\ch Guerra~\cite{Guerra-Sum} (see also Talagrand~\cite[Proposition 1.3.8]{Talagrand} where an independent proof of Latala is also mentioned) proved that for $h \neq 0$, the infimum is uniquely attained at $q = q (\beta, h)$ which satisfies
\begin{equation}
	\label{e:q}
	q=\EE\tanh^2(h+\beta\sqrt{q} Z).
\end{equation}
Here and in the following, $Z$ (under a probability $\sf{P}$ with associated expectation $\EE$) always denotes a standard Gaussian. For $h\neq 0$, the fixed point equation~\eqref{e:q} has a unique solution  which we denote in the sequel by $q$. \ech
A proof of Theorem~\ref{t:RS} based on an approach of Thouless-Anderson-Palmer (TAP for short)  \cite{tap} can be found in \cite{b2}. The critical temperature $\beta_0$ in Theorem~\ref{t:RS}  has then been improved in \cite{chr} using the same approach. Actually, $f(\beta, h) = RS(\beta, h)$ is believed to hold under the de Almeida-Thouless condition (AT for short), i.e.\ for $(\beta,h)$ with
\begin{equation}\label{e:AT}
	\beta^2\EE\frac{1}{\cosh^4(h+\beta\sqrt{q}Z)} \le 1,
\end{equation}
 but this problem is still open (however, Toninelli~\cite{To02} proved that when~\eqref{e:AT} is not satisfied, then the assertion of Theorem~\ref{t:RS} does not hold anymore). De Almeida and Thouless found the condition \eqref{e:AT} in 1978 in the context of an instability in the replica procedure \cite{AT} which is hard to make rigorous. 
 We also mention that Chen~\cite{Chen21} recently established the de Almeida-Thouless line as the transition curve between the replica symmetric and the replica symmetry breaking phases in a SK model with centered Gaussian external field.

To state our results, we first introduce the TAP free energy of the SK model. Analysis of the SK model in terms of the TAP equations was first given by \cite{tap}: shortening $\bar g_{ij}=\tfrac{1}{\sqrt 2}(g_{ij} + g_{ji})$, and for $\bfm=(m_i)\in[-1,1]^N$, \bch the TAP free energy \ech is given by 
\begin{equation}
	\label{e:FEN}
	\TAP_N(\bfm)=\frac{\beta}{\sqrt{N}}\sum_{\substack{i,j=1\\i<j}}^N  \bar g_{ij} m_i m_j + h\sum_{i=1}^N m_i + \frac{\beta^2}{4} N\left( 1-\frac1N\sum_{i=1}^N m_i^2 \right)^2 - \sum_{i=1}^N I(m_i),
\end{equation}
where for $ x\in[-1,1]$, 
\begin{equation}
	\label{e:tricks}
	I(x)=\frac{1+x}{2}\log(1+x) + \frac{1-x}{2}\log(1-x)=x\tanh^{-1}(x) - \log\cosh\tanh^{-1}(x)\,.
\end{equation}
The TAP free energy can be related to the free energy by a variational principle: Chen and Panchenko~\cite[Theorem 1]{chen} show that
\begin{equation}\label{e:chen}
	f(\beta,h)=\lim_{\epsilon\downarrow 0}\lim_{N\to\infty}\E\max N^{-1}\TAP_N(\bfm),
\end{equation}
where the maximum is over all $\bfm\in[-1,1]^N$ with $N^{-1}\sum_{i=1}^N m_i^2 \in[q_P-\epsilon,q_P+\epsilon]$, $q_P$ denoting the right edge of the support of the Parisi measure.
We also mention that an upper bound of the free energy in terms of the TAP free energy has recently been given by Belius~\cite{belius}. For the SK model with spherical spins, a variational principle for the TAP free energy has been proved in~\cite{kistler2}.

The TAP free energy can also be interpreted \bch non-rigorously as the power expansion up to second order of the Legendre transform of \ech the Gibbs potential of the SK model~\cite{plefka} (see~\cite{kistler} for further discussion).
A necessary condition of Plefka~\cite{plefka} for the convergence of the power expansion is that the magnetizations are in
\begin{equation}\label{Plefka1}
P^{1}_N := \{\bfm \in [-1,1]^N,\frac{\beta^2}{N}\sum_{i=1}^{N} \left(1-m_i^2\right)^2 < 1\}.
\end{equation}
$P^{1}_N $ is the set of magnetizations satisfying the so-called first Plefka condition. Before Plefka, this condition was also noted by Bray and Moore~\cite{bm} who investigated the Hessian matrix of the TAP free energy. For the stability of a diagrammatic expansion of the free energy, Sommers~\cite{sommers} also obtained condition~\eqref{Plefka1}.  There is no rigorous justification whether the first Plefka condition suffices for neglecting the higher-order terms (cf.\ also the discussion in~\cite{owen}). 

\bch If these higher-order terms can be neglected in Plefka's expansion, it is reasonable to expect concavity of the TAP free energy $\TAP_N$ in $\bfm$
as this functional approximates a Legendre transform. \ech Let us remark that \ech the TAP functional is not necessarily concave: we consider the Hessian
\begin{equation}\label{e:def-Hm}
	\bs H(\bfm):=\frac{\partial^2}{\partial m_i \, \partial m_j} \TAP_N(\bfm)
\end{equation}
at arbitrary magnetizations $\bfm\in[-1,1]^N$. \ch Denoting by $\lambda_1(\bs M)$ the largest eigenvalue of a real and symmetric matrix $\bs M$, we then \ech have:
\begin{teor}\label{t:pos-macr}
	There exists $\beta_0\in(0,1)$ such that for all $\beta>\beta_0$, $h\ne 0$, there exist $\epsilon>0$ and random $\bfm_N \in P^1_N$ such that
	\begin{equation}
		\lim_{N\to\infty}\bbP\left( \lambda_1\left( \bs H(\bfm_N)\right) >\epsilon \right)=1.
	\end{equation}
\end{teor}
This observation is proved in Section~\ref{s:p-pos-macr}. \ch Now the question arises whether concavity is also lost in the vicinity of the maximizer of $\TAP_N$, as the magnetization $\bfm_N$ for which Theorem~\ref{t:pos-macr} can be proved is somewhat arbitrary (see~\eqref{e:def-msign}) and might not be in the domain over which the maximum is taken in the variational principle~\eqref{e:chen}. \ech
 
The fixed points of the TAP equations~\cite{tap}
\begin{equation}\label{e:TAP}
	m_i= \tanh\left( h + \frac{\beta}{\sqrt{N}} \sum_{\substack{j=1\\j\neq i}}^N \bar g_{ij} m_j - \beta^2 \left(1-\frac1N\sum_{j=1}^N m_j^2\right) m_i \right),\quad i=1,\ldots,N
\end{equation}
are the critical points of the TAP free energy $\TAP_N$ (see also e.g.\ \cite{Talagrand, C10,FMM21,CFM23,Jacobian} for analysis of the TAP equations).
As we are not able to control these fixed points, we base our analysis on Bolthausen's algorithm \cite{b1,b2} which yields a sequence $\bfm^{(k)}\in[-1,1]^N$ of magnetizations that are considered as an approximation of the solutions of~\eqref{e:TAP}.
In~\cite{b1}, the magnetizations $\bfm^{(k)}$ are constructed by a two-step Banach algorithm: $m^{(0)}_i := 0$, $m^{(1)}_i:=\sqrt{q}$, and then iteratively 
\begin{equation}\label{e:two-step}
m^{(k+1)}_i = \tanh\left( h + \frac{\beta}{\sqrt{N}} \sum_{\substack{j=1\\j\neq i}}^N \bar g_{ij} m_j^{(k)}- \beta^2 (1-q) m_i^{(k-1)}\right), 
\end{equation}
for $k\geq 1$. \bch In the present paper, we use the algorithm from~\cite{b2}, which is a slight modification of the one in~\cite{b1}.
This definition of the magentizations $\bs m^{(k)}$ is recalled in Section~\ref{Bolthausen}. \ech
Bolthausen \cite{b1,b2} proves that such sequence of magnetisations converges, in the sense of~\eqref{e:B-conv} below, up to the AT-line. Precisely, by means of a sophisticated conditioning procedure which will be recalled in Section~\ref{Bolthausen}, Bolthausen shows that the iterates satisfy  
\begin{equation}\label{e:B-conv}
\lim_{k,l \to \infty} \lim_{N \to \infty} \frac{1}{N} \sum_{i=1}^N \E\left[ \left( m_i^{(k)} - m_i^{(l)} \right)^2 \right] = 0, 
\end{equation}
{\it provided} $(\beta, h)$ satisfy the AT-condition. \ch Under a high-temperature condition, \ech Chen and Tang \cite[Theorem~3]{CT} obtain that
\begin{equation}
	\lim_{k \to \infty} \lim_{N \to \infty} \frac{1}{N} \sum_{i=1}^N \E\left[ \left( \langle \sigma_i\rangle - m_i^{(k)} \right)^2 \right] = 0, 
\end{equation}
where $\langle\sigma_i\rangle$ denotes the Gibbs average of $\sigma_i$ under the Hamiltonian~\eqref{e:def-Hamiltonian}. It is crucial to emphasize that due to the factor $1/N$ in the distance~\eqref{e:B-conv} and the limit $N\uparrow +\infty$ first, and only in a second step $k,l \uparrow +\infty$, it is not clear whether the convergence of Bolthausen's approximate solutions is sufficiently strong for the interpretation of our results. Notwithstanding, the following suggests that Bolthausen's magnetizations are {\it good enough} when it comes to computing the limiting free energy within the TAP approximation: 
\begin{teor}\label{t:FEN0}
	For all $\beta>0$, $h\ne 0$ satisfying~\eqref{e:AT}, the TAP free energy evaluated at the Bolthausen approximate fixed points converges to the replica symmetric functional,  
 \begin{equation}\label{e:t:FEN0}
		\lim_{k\to\infty}\lim_{N\to\infty}N^{-1}\TAP_N(\bfm^{(k)})=RS(\beta,h) \, \text{in } L_1(\bbP).
	\end{equation}
\end{teor}
By the above, we shall henceforth refer to Bolthausen's magnetisations as {\it approximate solutions} of the TAP-equations.
The proof of Theorem~\ref{t:FEN0} is given in Section~\ref{s:TAP-FE}.
A result similar to Theorem~\ref{t:FEN0} occurs in Theorem 2 of Chen and Panchenko~\cite{chen}.
After the prepublication of this paper, the preprint of Gayrard~\cite{Gayrard} appeared where the almost sure convergence in Theorem~\ref{t:FEN0} is shown.
Under the AT condition~\eqref{e:AT}, Bolthausen's magnetizations $\bfm^{(k)}$ actually satisfy Plefka's first condition~\eqref{Plefka1} with high probability as $N\to\infty$: indeed, it follows from Lemma~\ref{l:LLN} below that 
 \begin{equation}\label{Plefka1lim}
\lim_{N \to \infty}\frac{\beta^2}{N}\sum_{i=1}^{N} \left(1-{m_i^{(k)}}^2\right)^2=\beta^2\EE\frac{1}{\cosh^4(h+\beta\sqrt{q}Z)} \quad \text{in } L_1(\bbP)\,.
\end{equation}
As a consequence, if the AT condition~\eqref{e:AT} holds with strict inequality, then with probability tending to $1$ as $N\to\infty$, Bolthausen's approximate solution satisfies the first Plefka condition, $\bfm^{(k)}\in P^1_N$.
That the AT condition and the first Plefka condition are related for suitable magnetizations was clear to Plefka \cite{plefka}.

We now investigate the concavity of the $\TAP_N(\bfm)$ functional in the Bolthausen magnetizations, that is, we study the Hessian $\bs H^{(k)}$ of the $\TAP$ free energy evaluated in $\bfm^{(k)}$,
\begin{equation}\label{e:def-Hk0}
	\bs H^{(k)}:=\frac{\partial^2}{\partial m_i \, \partial m_j} \TAP_N(\bfm) \Big|_{\bfm=\bfm^{(k)}}.
\end{equation}
We consider the weak limit of the empirical distribution of the eigenvalues $\lambda_i(\bs H^{(k)})$ ($i=1,\ldots,N$) of $\bs H^{(k)}$,
\begin{equation}
	\mu_{\bs H^{(k)}}=\frac1N \sum_{i=1}^N\delta_{\lambda_i(\bs H^{(k)})},
\end{equation}
which we show to be concentrated strictly below $0$ if the AT condition~\eqref{e:AT} holds with strict inequality, and to touch zero if~\eqref{e:AT} holds with equality. \ch This gives a spectral interpretation of the AT line. Similar observations have been made non-rigorously in \cite{chr2} (in the $N\to\infty$ limit), and are contained implicitly in \cite{plefka} (through relation~\eqref{Plefka1lim}). \ech
\begin{teor}
	\label{t:H0}
	For all $\beta>0$, $h\ne 0$ satisfying condition~\eqref{e:AT}, the empirical spectral distribution $\mu_{\bs H^{(k)}}$ converges weakly in distribution as $N\to\infty$ followed by $k\to\infty$ to a deterministic limiting measure~$\mu$. If~\eqref{e:AT} holds with strict inequality, then~$\mu(t,\infty)=0$ for some $t<0$. If~\eqref{e:AT} holds with equality, then~$\sup\{t\in\R:\mu(t,\infty)>0\}=0$.
\end{teor}
For the proof of Theorem~\ref{t:H0}, we use in Section~\ref{s:Cond} the explicit control of the weak independence between the disorder $(\bar g_{ij})$ and the approximate magnetizations $(m^{(k)}_i)$, which is given by Bolthausen's algorithm, and we conclude in Section~\ref{s:p-H0} using results from free probability which we recall in Section~\ref{s:Wigner}. \bch We remark that Theorem~\ref{t:H0} and its proof also pass through for magnetizations $(m_i)$ that are assumed to be independent (or sufficiently weakly dependent of $(\bar g_{ij})$: in other words, the correlation between $(\bar g_{ij})$ and $(m_i^{(k)})$ that is provided by Bolthausen's algorithm does not play a role at this level of accuracy. \ech
Theorem~\ref{t:H0} ensures that under the AT condition, no positive proportion of the eigenvalues of $\bs H^{(k)}$ becomes positive, in this sense, $\bs H^{(k)}$ does not lose concavity ``on a macroscopic scale''.
If and only if the AT condition holds with {\it strict} inequality, the right edge of the support of the weak limit of the spectrum is strictly smaller than zero, ensuring that {\it strict} concavity is not lost ``on a macroscopic scale''. \ch However, we are unable to show that outlier eigenvalues, which are too few to have positive mass and thus are not visible in the weak limit, do not lead to a loss of concavity ``on a microscopic scale'' for large $N$, $k$. Instead, we prove a rigorous interpretation of Plefka's second condition in Theorem~\ref{Plefka2} below. We remark that the TAP free energy for a Sherrington-Kirkpatrick model with ferromagnetically biased coupling is analysed in~\cite{CFM23,FMM21}. In particular, strong convexity of the corresponding TAP functional for sufficiently strong coupling is shown in~\cite{CFM23,FMM21} using free probability and the Kac-Rice formula. After the prepublication of this article, the preprint of Gayrard~\cite{Gayrard} appeared where the strict concavity of the TAP free energy in the SK model is shown for a region of the $\bfm$'s which comprises the Bolthausen approximations, and for $(\beta,h)$ in a region that does not correspond to the AT condition. In such a region of $(\beta,h)$, Gayrard~\cite{Gayrard} then characterizes the limiting free energy by a variational principle.

Besides condition $P^1_N$, which is related to the weak limit of the spectrum, Plefka~\cite{plefka} states a second condition
\begin{equation}\label{Plefka2}
	P^{2}_N := \{\bs m \in [-1,1]^N,\frac{2\beta^2}{N}\sum_{i=1}^{N} (m_i^2-m_i^4) < 1\}
\end{equation}
which he relates to the loss of concavity on a ``microscopic'' scale. However, Owen~\cite{owen} argues that Plefka's second condition is not necessary and comes from an incorrect assumption, namely that the disorder $(\bar g_{ij})$ and the magnetizations $(m_i)$ were independent. Other than for the weak limit of the spectrum in Theorem~\ref{t:H0}, we expect that weak dependence (as present in Bolthausen's ``approximate solutions'') between $(\bar g_{ij})$ and $(m_i)$ does change the condition for the limiting largest eigenvalue to be positive. To illustrate the role of Plefka's second condition, we consider in the following theorem magnetizations $\bs m_N$  that differ from Bolthausen's approximate solutions as they are assumed to be independent of the disorder $\bar{\bs g}$. In this setting, we show rigorously that Plefka's second condition is equivalent to all outlier eigenvalues of the Hessian to be negative. For $\epsilon>0$, we denote by
\begin{equation}\label{e:p2eps}
	\bar P^{2,\epsilon}_N :=	\{\bs m \in [-1,1]^N,\frac{2\beta^2}{N}\sum_{i=1}^{N} (m_i^2-m_i^4) > 1+\epsilon\}
\end{equation}
sets of magnetizations that do not satisfy Plefka's second condition.
\begin{teor}\label{t:pos}
	Let $\beta>0$, $h\ne 0$ such that the AT condition~\eqref{e:AT} is satisfied with strict inequality. Let $\bs m_N$ be $[0,1]^N$-valued random vectors that are independent of $\bs{\bar g}$ and satisfy $\frac1N\sum_{i=1}^N \delta_{m_{N,i}}\overset{w}\to \mathcal L(\tanh(h+\beta\sqrt{q} Z))$ as $N\to\infty$, where $Z$ is a standard Gaussian random variable.
	\begin{enumerate}
		\item[(i)]
	If $\bs m_N$ takes values only in $P_N^1\cap P_N^2$, then
	\begin{equation}
		\lim_{N\to\infty}\bbP\left(\lambda_1\left(\bs H\left(\bs m_N\right)\right)\le 0\right)=1.
	\end{equation}
	\item[(ii)] Conversely, if $\bs m_N$ takes values only in $P_N^1\cap \bar P_N^{2,\epsilon}$ for some $\epsilon>0$, then there exists $\epsilon'>0$ such that
	\begin{equation}
		\lim_{N\to\infty}\bbP\left(\lambda_1\left(\bs H\left(\bs m_N\right)\right)\ge \epsilon' \right)=1.
	\end{equation}
	\end{enumerate}
\end{teor}
The proof of Theorem~\ref{t:pos} is given in Section~\ref{s:p-pos-ind} and relies on a generalization of the convergence result of Capitaine et al.\ \cite{CDFF} for the largest eigenvalue of perturbed complex Wigner matrices which we state in Lemma~\ref{l:maxgen} below.
\begin{rem}
	As we may use that $\frac1N\sum_{i=1}^N \delta_{m_{N,i}}\overset{w}\to \mathcal L(\tanh(h+\beta\sqrt{q} Z))$ in Theorem~\ref{t:pos}, it follows that $\bs m_N\in P^1_N$ for all sufficiently large $N$ if $\beta>0$, $h\ne 0$ satisfy the AT condition~\eqref{e:AT} with strict inequality. Analogously, it follows that $\bs m_N\in \bar P^{2,\epsilon}_N$ for all sufficiently large $N$ if $\beta>0$, $h\ne 0$ satisfy
	\begin{equation}\label{e:rem-P2}
		2\beta^2\EE\left(\tanh^2(h+\beta\sqrt{q}Z)- \tanh^4(h+\beta\sqrt{q}Z)\right) > 1+\epsilon.
	\end{equation}
	Similarly, we have $\bs m_N\in P^{2}_N$ for all sufficiently large $N$ if $\beta>0$, $h\ne 0$ are such that
	\begin{equation}\label{e:Plefka2}
		2\beta^2\EE\left(\tanh^2(h+\beta\sqrt{q}Z)- \tanh^4(h+\beta\sqrt{q}Z)\right) < 1.
	\end{equation}
	
		It can be seen numerically that the set of $(\beta,h)$ which satisfy both~\eqref{e:AT} and~\eqref{e:rem-P2} for some $\epsilon>0$ is non-empty.
	\begin{figure}[h]
		\begin{minipage}[c]{0.45\textwidth}
			\includegraphics[height=6cm, width = 8cm, trim={0cm 0cm 0 0},clip]{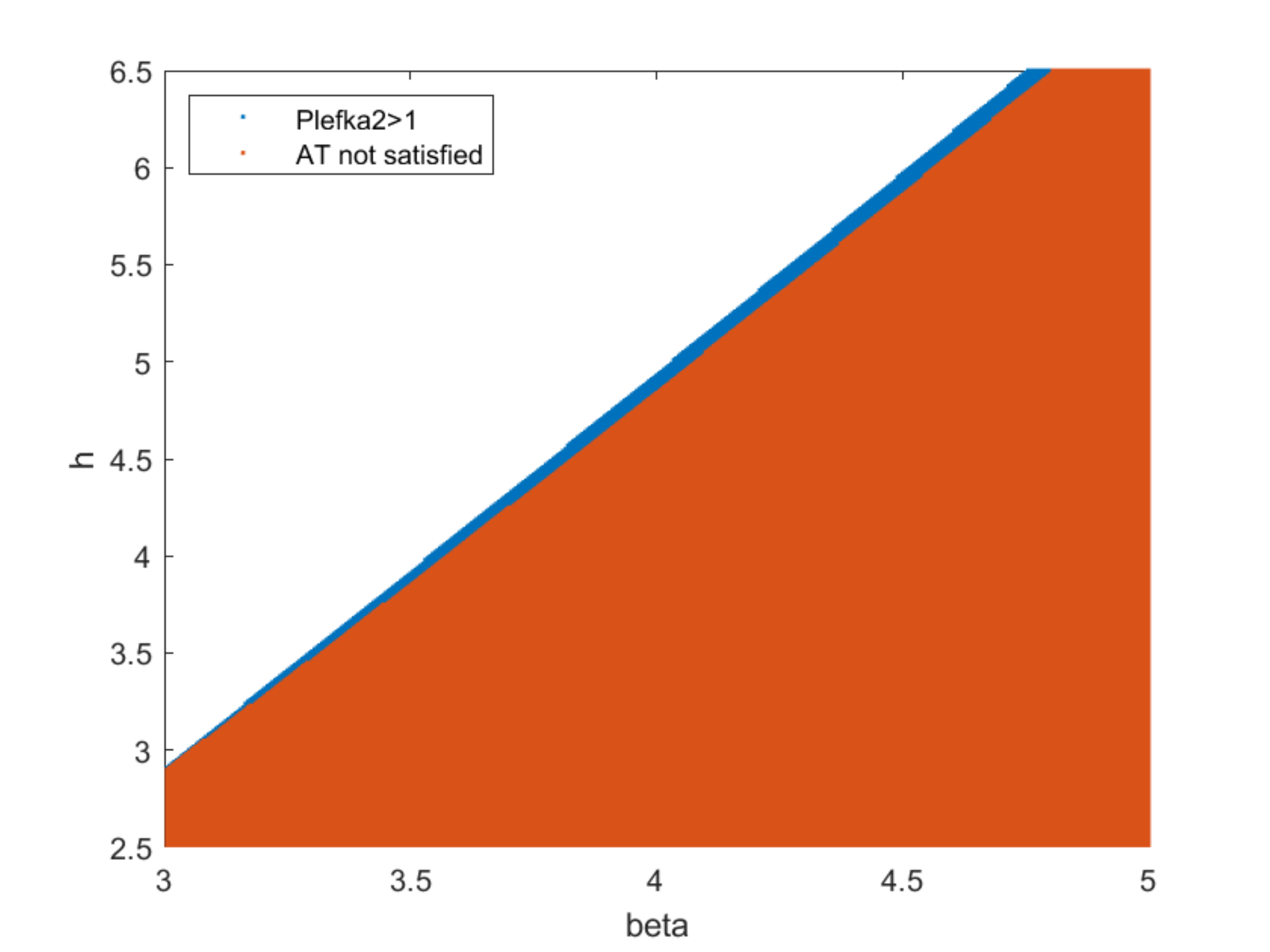}
		\end{minipage}
		\begin{minipage}[c]{0.55\textwidth}
			\caption{Phase diagram of inverse temperature $\beta>0$ and external field $h>0$. The region in which the AT condition~\eqref{e:AT} is not satisfied is depicted in red. The region in which~\eqref{e:AT} holds but~\eqref{e:Plefka2} does not hold is depicted in blue.
			} \label{figure1}
		\end{minipage}
	\end{figure}	
\end{rem}

\ech

\section{Bolthausen's iterative procedure} \label{Bolthausen}

We now recall the algorithm from~\cite{b2} which we will use throughout the paper. Also throughout the paper, we will assume that $\beta>0$ and $h\ne 0$.
A scalar product on $\R^N$ is given by $\langle \bs x,\bs y\rangle:=N^{-1}\sum_{i=1}^Nx_iy_i$ with associated norm $\|\bs x\|:=\sqrt{\langle \bs x, \bs x\rangle}$. Furthermore,
$\bs x\otimes \bs y:=N^{-1}(x_iy_j)_{ij}$, and for a matrix $\bs A\in\R^{N\times N}$, we denote its symmetrization by $\bar{\bs A}:=\tfrac{1}{\sqrt 2}(\bs A+ \bs A^T)$.

Let $\bfg =(g_{ij})_{i,j=1,\ldots,N}$ be an array of independent centered Gaussians with variance $1$. The interaction matrix will be its symmetrization $\bar\bfg$, normalized by $N^{-1/2}$.
Let $\psi:[0,q]\to[0,q]$ be defined by
\begin{equation}
	\psi(t)= \EE\tanh\left( h+ \beta \sqrt{t} Z + \beta\sqrt{q-t} Z' \right)\tanh\left( h+ \beta \sqrt{t} Z + \beta\sqrt{q-t} Z'' \right),
\end{equation}
where $Z,Z',Z''$ are independent standard Gaussians. Then set
\begin{equation}
	\gamma_1:=\EE\tanh\left( h+ \beta\sqrt{q} Z \right), \quad \rho_1:=\sqrt{q}\gamma_1
\end{equation}
and
\begin{equation}
	\rho_k:= \psi(\rho_{k-1}),\quad \gamma_k:=\frac{\rho_k-\sum_{j=1}^{k-1} \gamma_j^2}{\sqrt{q-\sum_{j=1}^{k-1}\gamma_j^2}}.
\end{equation}
Let $\bfg^{(1)}:=\bfg$, $\phi^{(1)}={\mathbf 1}$, $\bfm^{(1)}=\sqrt{q}{\mathbf 1}$. 	
With the shorthand $\Gamma_k^2:=\sum_{j=1}^k\gamma_j^2$, we set recursively for $k\in\N$
\begin{equation}
	\xi^{(k)}=\tfrac{1}{\sqrt N}\bfg^{(k)}\phi^{(k)},\quad \eta^{(k)}=\tfrac{1}{\sqrt N}{\bfg^{(k)}}^T\phi^{(k)},\quad \zeta^{(k)}=\frac{1}{\sqrt 2}\left(\xi^{(k)} + \eta^{(k)}\right),
\end{equation}
\begin{equation}\label{e:h-it}
	\bfh^{(k+1)}=h\mathbf{1}+\beta\sum_{s=1}^{k-1}\gamma_s\zeta^{(s)}+\beta\sqrt{q-\Gamma^2_{k-1}}\zeta^{(k)},
\end{equation}
\begin{equation}
\bfm^{(k+1)}=\tanh(\bfh^{(k+1)}),
\end{equation}
moreover $\{\phi^{(1)},\ldots,\phi^{(k+1)}\}$ as the Gram-Schmidt orthonormalization of $\{\bfm^{(1)},\ldots,\bfm^{(k+1)}\}$,
\begin{equation}\label{e:def-phi}
	\phi^{(k+1)}=\frac{ \bfm^{(k+1)} - \sum_{s=1}^k \langle \phi^{(s)},\bfm^{(k+1)}\rangle\phi^{(s)}}{\big\|\bfm^{(k+1)} - \sum_{s=1}^k \langle \phi^{(s)},\bfm^{(k+1)} \rangle\phi^{(s)}\big\|},
\end{equation}
and the modifications of the interaction matrices
\begin{equation}\label{e:def-gk}
	\bfg^{(k+1)}=\bfg^{(k)} - \sqrt{N}\rho^{(k)} ,
\end{equation}
where
\begin{equation}\label{e:def-rho}
	\rho^{(k)}=\xi^{(k)}\otimes\phi^{(k)}+\phi^{(k)}\otimes\eta^{(k)}-\big\langle\phi^{(k)},\xi^{(k)}\big\rangle\big(\phi^{(k)}\otimes\phi^{(k)}\big).
\end{equation}
By Lemma~2b of~\cite{b2}, we have
\begin{equation}\label{e:2bb}
	\sum_{s=1}^\infty\gamma_s^2=q.
\end{equation}
Noting that $\{\phi^{(s)}\}_{s\leq k}$ are orthonormal with respect to $\langle\cdot,\cdot \rangle$, we define
\begin{equation}\label{e:def-P}
	P^{(k)}_{ij} = \frac1N \sum_{s=1}^{k} \phi^{(s)}_i\phi^{(s)}_j,
\end{equation}
and one readily checks that $\boldsymbol{P}^{(k)}$ is an orthogonal projection.
Furthermore, let
\begin{equation}
	\mathcal G_k = \sigma\left(\xi^{(m)},\eta^{(m)}:\: m\le k\right).
\end{equation}
Then $\zeta^{(k)}$ is $\mathcal G_k$-measurable and $\bfm^{(k)}$ is $\mathcal G_{k-1}$-measurable.
Moreover, by Proposition~4 of~\cite{b2}, $\bfg^{(k)}$ is centered Gaussian under $\bbP\left(\cdot\mid \mathcal G_{k-1}\right)$ with covariances given by
\begin{equation}\label{e:gCov}
	V^{(k)}_{ij,st}:=\E\Big( g^{(k)}_{ij} g^{(k)}_{st}\,\Big|\, \mathcal G_{k-1} \Big) = Q^{(k-1)}_{is}Q^{(k-1)}_{jt},
\end{equation}
where $\boldsymbol{Q}^{(k)}=(Q^{(k)}_{ij})_{ij\le N}=\boldsymbol{1}-\boldsymbol{P}^{(k)}$. As we show in Lemma~\ref{l:Cov} below, this covariance matrix itself is a projection.\footnote{As we want this covariance to be a projection, we define the entries of $\bfg$ with unit variance, while~\cite{b2} defines them with variance $1/N$. As a consequence, we have to carry along the scaling factor $N^{-1/2}$ when $\bfg$ is used.}

If $X_N$, $Y_N$ are two sequences of random variables, we write
\begin{equation}
	X_N \simeq Y_N,
\end{equation}
if there exists a constant $C > 0$, depending possibly on other parameters, but not on $N$, with
\begin{equation}
	\bbP\big(|X_N-Y_N| >t\big)\leq C\rme^{-t^2N/C}.
\end{equation}
$X_N \simeq Y_N$ in particular implies $X_N - Y_N \to 0$ in $L_p(\bbP)$ for every $p>0$ as $N \to \infty$.
By Proposition~6 of~\cite{b2}, we have
\begin{equation}\label{e:mq}
\|\bfm^{(k)}\|\simeq q	
\end{equation}
for each $k\in\N$.

\ch We will also use the following lemma. We recall the definition of $\bs h^{(k)}$ from~\eqref{e:h-it}.
\begin{lem}[Law of large numbers, Lemma~14 of~\cite{b2}] \label{l:LLN}
	For any Lipschitz continuous function $f: \R\to\R$ with $|f(x)| \le C (1+|x|)$
for some constant $C<\infty$, and any $k\ge 2$, we have for $\beta>0$, $h\ne 0$ satisfying~\eqref{e:AT} that
\begin{equation}\label{e:LLN-f}
	\lim_{N\to\infty}\frac 1N\sum_{i=1}^N f\left(h^{(k)}_i\right)
	=\EE f(h+\beta\sqrt{q} Z),
\end{equation}
in $L_1(\bbP)$.
\end{lem}\ech

\section{Replica symmetric formula for the TAP free energy}
\label{s:TAP-FE}

To prove Theorem \ref{t:FEN0}, we will use the following lemma. \bch Here we recall that  $\bar{\bs g}:=\tfrac{1}{\sqrt 2}(\bs g+ \bs g^T)$ where $\bs g\in\R^{N\times N}$ is a matrix whose entries $(g_{ij})_{i,j=1,\ldots,N}$ are independent standard Gaussians, and $\bs m^{(k)}$ are the corresponding magnetizations provided by Bolthausen's algorithm from Section~\ref{Bolthausen}.\ech
\begin{lem}\label{l:b1b2}
	Let
	\begin{equation}\label{e:def-Delta}
		\Delta^{(k)}=\tanh^{-1}(\bfm^{(k)}) - h\mathbf{1} - \frac{\beta}{\sqrt N}\bar \bfg \bfm^{(k)} + \beta^2(1-q) \bfm^{(k)},
	\end{equation}
	and assume that $\beta>0$, $h\ne 0$ satisfies the AT condition~\eqref{e:AT}. Then,
	\begin{equation}
		\|\Delta^{(k)}\| \to 0\quad\text{in }L_2(\bbP)\quad\text{as } N\to\infty\text{ followed by } k\to\infty\,.
	\end{equation}
\end{lem}
\begin{proof}
		Let us write $X\stackrel{N,k}{\sim} Y$ if $\|X-Y\|\to 0$ in $L_2(\bbP)$ as $N\to\infty$ followed by $k\to\infty$, and $\stackrel{N}{\sim}$ if the norm vanishes already as $N\to\infty$. From~\eqref{e:h-it}, we have
		\begin{equation}\label{e:p-lrec-1}
			\beta^{-1}\Delta^{(k)}=\sum_{s=1}^{k-2}\gamma_s\zeta^{(s)} + \sqrt{q-\Gamma_{k-2}^2}\zeta^{(k)} - \frac{1}{\sqrt N} \bar \bfg\bfm^{(k)} + \beta(1-q)\bfm^{(k)}\,.
		\end{equation}
		We note that the second term on the right-hand side of~\eqref{e:p-lrec-1} is $\stackrel{N,k}{\sim}0$ by Lemmas~2 and~15a of~\cite{b2}. Thus it holds
		\begin{equation}\label{e:p-lrec-1-bonus}
			\beta^{-1}\Delta^{(k)}\stackrel{N,k}{\sim}\sum_{s=1}^{k-2}\gamma_s\zeta^{(s)} - \frac{1}{\sqrt N} \bar \bfg\bfm^{(k)} + \beta(1-q)\bfm^{(k)}\,.
		\end{equation}
		We will now rewrite the second term of \eqref{e:p-lrec-1-bonus}. From~\eqref{e:def-rho}, we obtain
		\begin{equation}\label{e:ggk}
			\bar \bfg = \bar \bfg^{(k)} + \sqrt{\tfrac{N}{2}}\sum_{s=1}^{k-1}\big(\rho^{(s)}+{\rho^{(s)}}^T\big),
		\end{equation}
		which we use together with~\eqref{e:def-rho} in
		\begin{equation}
			\tfrac{1}{\sqrt{N}}\bar \bfg\bfm^{(k)} = \tfrac{1}{\sqrt{N}}\bar \bfg^{(k)}\bfm^{(k)} + \sum_{s=1}^{k-1}  \Big[\zeta^{(s)}\langle \phi^{(s)},\bfm^{(k)}\rangle
			+ \phi^{(s)}\langle \zeta^{(s)}, \bfm^{(k)}\rangle
			- \sqrt2 \langle \phi^{(s)},\xi^{(s)}\rangle\phi^{(s)}\langle\phi^{(s)} ,\bfm^{(k)} \rangle\Big]\,.
		\end{equation}
		The expression on the right-hand side of the last display is
		\begin{equation}\label{e:lrec-2}
			\stackrel{N,k}{\sim}	 \sum_{s=1}^{k-1}  \gamma_s \zeta^{(s)}
			+  \sum_{s=1}^{k-2} \phi^{(s)}\beta\gamma_s(1-q)
			+ \phi^{(k-1)}\beta(1-q)\sqrt{q-\Gamma_{k-2}^2}
			-\sqrt2  \sum_{s=1}^{k-1}\langle \phi^{(s)},\xi^{(s)}\rangle\phi^{(s)}\gamma_s,
		\end{equation}
		by Proposition 6, Lemmas~13 and~16 of~\cite{b2}.
		As $\|\phi^{(s)}\|=1$, the last term in~\eqref{e:lrec-2} is $\stackrel{N}{\sim} 0$ by Lemma~11 of~\cite{b2}.
		
		Replacing $\tfrac{1}{\sqrt{N}}\bar\bfg\bfm^{(k)}$ in~\eqref{e:p-lrec-1-bonus} by~\eqref{e:lrec-2}, after cancellations we obtain 
		\begin{equation}\label{e:a6}
			\beta^{-1} \Delta^{(k)} \stackrel{N,k}{\sim} -\gamma_{k-1} \zeta^{(k-1)}+\beta(1-q)\bigg(\bfm^{(k)}
			-\sum_{s=1}^{k-2} \phi^{(s)}\gamma_s
			- \phi^{(k-1)}\sqrt{q-\Gamma_{k-2}^2}\bigg).
		\end{equation}
		By Lemmas~2 and~15a of~\cite{b2}, the first term on the r.h.s.\ of~\eqref{e:a6} vanishes and we obtain
		\begin{equation}
			\beta^{-1}\Delta^{(k)}\stackrel{N,k}{\sim}\beta(1-q)\bigg(\bfm^{(k)}
			-\sum_{s=1}^{k-2} \phi^{(s)}\gamma_s
			- \phi^{(k-1)}\sqrt{q-\Gamma_{k-2}^2}\bigg),
		\end{equation}
		where the $\|\cdot\|$ norm of the r.h.s.\ is bounded by
		\begin{equation}\label{e:a8}
			\beta(1-q) \left( \bigg\|\bfm^{(k)}
			-\sum_{s=1}^{k-2} \phi^{(s)}\gamma_s\bigg\|
			+\sqrt{q-\Gamma_{k-2}^2} \bigg\|\phi^{(k-1)}\bigg\|\right).
		\end{equation}
		By~\eqref{e:2bb}, we have that $\lim_{k\to\infty}\sqrt{q-\Gamma_{k-2}^2}=0$, recalling that $\|\phi^{(k-1)}\|=1$, the last term in the brackets on the r.h.s.\ of \eqref{e:a8} vanishes. As for the first term in the brackets, using the fact that $\{\phi^{(s)}\}$ is an orthonormal basis, it holds
		\beq
		\bigg\|\bfm^{(k)}	-\sum_{s=1}^{k-2} \phi^{(s)}\gamma_s\bigg\|^2=\bigg\|\bfm^{(k)}\bigg\|^2+\sum_{s=1}^{k-2} \gamma_s^2-2\sum_{s=1}^{k-2} \gamma_s \langle \bfm^{(k)},\bs \phi^{(s)}\rangle,
		\eeq
		By Proposition~6 of~\cite{b2} together with~\eqref{e:2bb} implies that the $\lim_{k\to\infty}\lim_{N\to\infty}$ of r.h.s.\ of the latter is equal to 0.			
\end{proof}
We are now ready to prove the convergence of the TAP functional to the replica-symmetric free energy:
\begin{proof}[Proof of Theorem~\ref{t:FEN0}]
	To see how this goes, we first reformulate \eqref{e:FEN} with the help of \eqref{e:tricks} with the right scaling,
	\begin{multline}\label{e:p-FEN-exp}
		N^{-1}\TAP_N(\bfm^{(k)})
		=\tfrac\beta2  N^{-3/2}\sum_{i\ne j}^N \bar g_{ij} m^{(k)}_i m^{(k)}_j + hN^{-1}\sum_{i=1}^N m^{(k)}_i
		- N^{-1}\sum_{i=1}^N \tanh^{-1}(m^{(k)}_i)m^{(k)}_i\\
		+ N^{-1}\sum_{i=1}^N \log\cosh\tanh^{-1}(m^{(k)}_i)
		+\frac{\beta^2}{4} \left( 1-\frac1N\sum_{i=1}^N {m^{(k)}_i}^2 \right)^2.
	\end{multline}
	By Lemma~\ref{l:LLN}, the terms in the second line of the latter converge in $L_1(\bbP)$ to the r.h.s.\ of~\eqref{e:t:FEN0} as $N\to\infty$. 
	
	It remains to show that the sum of the first three terms on the r.h.s.\ of~\eqref{e:p-FEN-exp} converges to $0$ in $L_1(\bbP)$ as $N\to\infty$, followed by $k\to\infty$. By Lemma~\ref{l:LLN}, first note that the limit in $L_1(\bbP)$ of the second and third term on the right-hand side of~\eqref{e:p-FEN-exp} is
	\beq\label{e:TAPFEintp}\bea
	\lim_{N \to \infty}  N^{-1}\sum_{i=1}^N \left(h-\tanh^{-1}(m^{(k)}_i)\right)m^{(k)}_i&=-\EE\left(\beta\sqrt{q} Z \tanh\left(h+\beta\sqrt{q} Z \right)\right)\\
	&=-\beta^2 q (1-q),
	\eea\eeq
	the last line by combining a simple integration by parts with \eqref{e:q}. It only remains to prove that the first term on the r.h.s.\ \eqref{e:p-FEN-exp} tends to $\beta^2 q (1-q)$. By Lemma~\ref{l:b1b2}, it holds
	\begin{equation}
		\frac{\beta}{\sqrt N} \sum_{j:\: j\ne i} \bar g_{ij} m^{(k)}_j = - h + \beta^2 (1-q) m^{(k)}_i + \tanh^{-1}(m^{(k)}_i) - \Delta_i^{(k)}-\frac{\beta}{\sqrt N} \bar g_{ii} m^{(k)}_i ,
	\end{equation}
	 Multiplying the latter by $\frac{m^{(k)}_i}{2N}$ and taking the sum over $i$ yield
	\begin{multline}\label{e:TAPFE-lin-p1-1}
		\tfrac12 \beta N^{-3/2}\sum_{i\ne j}^N \bar g_{ij} m^{(k)}_i m^{(k)}_j
		=-\tfrac12 hN^{-1}\sum_{i=1}^Nm^{(k)}_i +\tfrac12\beta^2(1-q)N^{-1}\sum_{i=1}^N{m^{(k)}_i}^2\\
		+\tfrac12 N^{-1}\sum_{i=1}^N m_i^{(k)}\tanh^{-1}(m_i^{(k)})-\tfrac12 N^{-1}\sum_{i=1}^N m_i^{(k)}\Delta^{(k)}_i -\frac{\beta}{2N\sqrt N} \sum_{i=1}^N \bar g_{ii} {m^{(k)}_i}^2.
	\end{multline}
	The last term on the r.h.s.\ of \eqref{e:TAPFE-lin-p1-1} tends to $0$ in $L_2(\bbP)$ as $N\to\infty$ as
	\begin{equation}\label{details}
		\E \frac1N \sum_{i=1}^N \left|\frac{\beta}{\sqrt{N}}\bar g_{ii}  m_i^{(k)}\right|^2 \leq
		\frac{\beta}{N^2}\sum_{i=1}^N \E \bar g_{ii}^2 = \frac{\beta}{N}.
	\end{equation}
	 Combining Cauchy-Schwarz with Lemma~\ref{l:b1b2} and~\eqref{e:mq}, we have that the second last term on the r.h.s.\ of \eqref{e:TAPFE-lin-p1-1} tends to $0$ in $L_2(\bbP)$ as $N\to\infty$ followed by $k\to\infty$.
	The sum of the remaining terms on the r.h.s.\ of \eqref{e:TAPFE-lin-p1-1} converges, as $N\to\infty$ in $L_1(\bbP)$ by Lemma~\ref{l:LLN} to
	\begin{multline}
		-\tfrac12 h \EE\left(\tanh\left(h+\beta\sqrt{q} Z \right)\right)
		+\tfrac12 \beta^2 (1-q)\EE\left(\tanh^2\left(h+\beta\sqrt{q} Z \right)\right)\\
		+\tfrac{1}{2}\EE\left(\tanh\left(h+\beta\sqrt{q} Z \right)\left(h+\beta\sqrt{q} Z \right)\right)=\beta^2q(1-q),
	\end{multline}
	again using~\eqref{e:TAPFEintp} and~\eqref{e:q}.
	All in all, we obtain that
	\begin{equation}
		\lim_{N\to\infty,\, k\to\infty}\tfrac12 \beta N^{-\frac{3}{2}}\sum_{i\ne j}^N \bar g_{ij} m^{(k)}_i m^{(k)}_j=\beta^2q(1-q)  , \text{in }  L_1(\bbP).
\end{equation}	
We proved that the first term on the r.h.s.\ of~\eqref{e:p-FEN-exp} tends to $\beta^2 q (1-q)$ and the assertion of Theorem~\ref{t:FEN0} follows.
\end{proof}

\section{Gaussian orthogonal ensemble}
\label{s:Wigner}
As a tool to study the Hessian of the TAP free energy functional, we record some known facts about the Gaussian orthogonal ensemble (GOE). A GOE with variance $\sigma^2>0$ is a real symmetric random matrix $\bs X$ with centered Gaussian entries of variance $\sigma^2$ off the diagonal, variance $2\sigma^2$ on the diagonal, and the entries $(X_{ij})_{1\le i\le j\le N}$ being independent.
The matrix $(\beta N^{-1/2} \bar g_{ij})_{i,j=1,\ldots, N}$ is a GOE with variance $\beta^2/N$. Thus, by Wigner's Theorem (see e.\,g.\ Theorem~2.1.1 in~\cite{AGZ}), its empirical spectral distribution converges weakly in probability to the semicircle law $\mu_\beta$ which is defined by its density
\begin{equation}\label{e:sc}
	\frac{\rmd \mu_\beta(x)}{\rmd x}=1_{[-2\beta,2\beta]}(x)\sqrt{4\beta^2-x^2}\,.
\end{equation}
Also, the largest eigenvalue $\lambda_1(\beta N^{-1/2} \bar{\bfg})$ converges a.\,s.\ to $2\beta$ (see e.\,g.~Theorem 1.13 of~\cite{tv10}).

For each real symmetric \bch (or Hermitian) \ech matrix $\bs M$ of size $n$, we denote the enumeration of its eigenvalues in non-increasing order by $\lambda_1(\bs M)\ge \ldots \ge \lambda_n(\bs M)$, and its empirical spectral distribution by
\begin{equation}
	\mu_{\bs M}:=\frac1N\sum_{i=1}^n \delta_{\lambda_i(\bs M)}.
\end{equation}
We recall that the Frobenius norm of a matrix $\bs M$ of size $n$ is defined by $\|\bs M\|_{\rm F}=(\sum_{i,j=1}^N |M_{ij}|^2)^{1/2}$.
The following standard result, for which we refer to Exercises~2.4.3 and~2.4.4 of~\cite{tao}, states that the limiting empirical spectral distributions of random matrices are invariant under additive perturbations in the prelimiting sequence that have either small rank or small Frobenius norm.
\begin{lem}\label{l:sp}
	Let $\bs M_n$ and $\bs N_n$ be random Hermitian matrices of size $n$ such that the empirical spectral distribution of $\bs M_n$ converges weakly a.\,s.\ to a probability measure~$\mu$. Suppose that at least one of the following conditions holds true:
	\begin{enumerate}[label=(\roman{*}),ref=(\roman{*})]
		\item\label{i:sp-F} $n^{-1} \|\bs N_n\|_{\rm F}^2\to 0$  a.\,s. ,
		\item\label{i:sp-r} $n^{-1}{\rm rank}(\bs N_n) \to 0$ a.\,s. .
	\end{enumerate}
	Then the empirical spectral distribution of $\bs M_n+\bs N_n$ converges to $\mu$ weakly a.\,s. .
\end{lem}
Similarly, for the largest eigenvalue we have:
\begin{lem}\label{l:F-max}
	Let $\bs M_n$ and $\bs N_n$ be random Hermitian matrices of size $n$ such that the largest eigenvalue $\lambda_1(\bs M_n)$ of $\bs M_n$ converges a.\,s.\ to a limit $\lambda_1$ as $n\to\infty$. Suppose that $\|\bs N_n\|_{\rm F}^2\to 0$  a.\,s. Then also the largest eigenvalue
	$\lambda_1(\bs M_n + \bs N_n)$ of $\bs M_n + \bs N_n$ converges to the same limit $\hat\lambda_1$ almost surely.
\end{lem}
\begin{proof}
\bch By the Weyl inequality (see e.g.\ (1.54) in~\cite{tao}),
	\begin{equation}
		\lambda_1(\bs M_n + \bs N_n)\le \lambda_1(\bs M_n) + \lambda_1(\bs N_n).
	\end{equation}
	By the dual Weyl inequality (see e.g.\ Exercise~1.3.5 in~\cite{tao}),
	\begin{equation}
		\lambda_1(\bs M_n + \bs N_n)\ge \lambda_1(\bs M_n) + \lambda_n(\bs N_n).
	\end{equation}
The assertion follows as $\lambda_1(\bs N_n)\le \|\bs N_n\|_{\rm F}$ and $\lambda_n(\bs N_n)\ge - \|\bs N_n\|_{\rm F}$. \ech
\end{proof}
\subsection{Free convolution}
First we state a definition of the free convolution (see~\cite{Biane,VDN,NicaSpeicher}). The Stieltjes transform of a probability measure $\mu$ on $\R$ is defined by
\begin{equation}\label{e:def-Stieltjes}
	g_\mu(z)=\int\frac{\rmd \mu(x)}{z-x}
\end{equation}
which is analytic in $\C\setminus\supp\mu$. It can be shown \bch that the Stieltjes transform characterizes the measure $\mu$ uniquely, and \ech that there exists a domain $D$ on which $g_\mu$ is univalent. Denoting by $K_\mu$ the inverse function of $g_\mu$ defined on $g_\mu(D)$, the R-transform of $\mu$ is defined on $g_\mu(D)$ by
\begin{equation}
	R_\mu(z)=K_\mu(z)-\frac{1}{z}\,.
\end{equation}
Free probability theory shows that for probability measures $\mu,\lambda$ on $\R$, there exists a unique probability measure $\kappa$ with
\begin{equation}
	R_\kappa=R_\lambda+R_\mu
\end{equation}
on a domain on which these three functionals are defined. The measure $\kappa$ is denoted by $\lambda\boxplus\mu$ and called the free (additive) convolution of $\lambda$ and $\mu$.

The following result ensures that limiting spectral distribution of a sum of a GOE and a deterministic matrix whose spectral distribution weakly converges is given by a free additive convolution with the semicircle law. The support of this free convolution is analyzed in Lemma~\ref{l:suppfc} below.
\begin{lem}[]
	\label{l:addconv}
	For $n\in\N$, let $\bs X_n$ be a GOE with unit variance, and let $\bs A_n$ be a deterministic real and symmetric matrix, each of size $n$, such that the empirical spectral distribution $\mu_{\bs A_n}$ converges weakly to some probability measure $\nu$ on $\R$ as $n\to\infty$. Then, for each $\sigma>0$, the empirical spectral distribution of $\sigma n^{-1/2}\bs X_n+\bs A_n$ converges weakly almost surely to $\mu_\sigma\boxplus\nu$. \bch  The Stieltjes transform of $\mu_\sigma\boxplus\nu$ is the unique solution $g$ of the following functional equation:
\begin{equation}\label{e:g-addconv}
\forall z \, \in {\C}^{+},\quad g(z)=g_\nu(z-\sigma^2 g(z)).
\end{equation}\ech
\end{lem}
\begin{proof}
	This is a standard result from free probability theory, see for example Theorem~5.4.5 in~\cite{AGZ}. \bch The functional equation~\eqref{e:g-addconv} is solved \ech by the Stieltjes transform of the limiting spectral distribution of $\sigma n^{-1/2}\bs X_n+\bs A_n$, \bch see e.g.\ Lejay and Pastur~\cite{pastur}, p.12. The functional equation~\eqref{e:g-addconv} has a unique solution, see Pastur~\cite{pastur1}, p.69. \ech We conclude with the fact that the Stieltjes transform of $\mu_\sigma\boxplus\nu$ \bch also solves the functional equation~\eqref{e:g-addconv}, \ech see e.g.\ Proposition~2.1 in~\cite{CDFF}.
\end{proof}
We will also use the following version of a result of Capitaine et al.\ \cite{CDFF} for the largest eigenvalue. For $\sigma>0$ and a probability measure $\nu$ on $\R$, let
\begin{equation}\label{e:def-Hz}
	H_{\sigma,\nu}(z):=z+\sigma^2 g_\nu(z)
\end{equation}
and
\begin{equation}\label{e:def-O}
	\mathcal O_{\sigma,\nu} := \{u\in \R\setminus \supp \nu\:: H_{\sigma,\nu}'(u)>0\}
\end{equation}
where $g_{\nu}$ denotes the Stieltjes transform defined as in~\eqref{e:def-Stieltjes}.
\begin{lem}[cf.\ \cite{CDFF}, Theorem 8.1]\label{l:maxgen}
	Let $\sigma>0$, let $\bs X_N$ be a GOE with unit variance, and let $\bs A_N$ be a deterministic real and symmetric matrix. Assume that the empirical spectral distribution $\mu_{\bs A_N}$ converges weakly to a probability measure $\nu$ on $\R$ as $N\to\infty$, and that there exists $s\in\R$ with $\nu(s,\infty)=0$.
	Also, suppose that there exist an integer $r\ge 2$ and $\theta\in\ch\R\setminus\supp\nu\ech$ with $\lim_{N\to\infty}\lambda_1(\bs A_N)=\theta$ and
	\begin{equation}\label{e:ass-bulk}
		\max_{j=r,\ldots,N}\rmd(\lambda_j(\bs A_N), \supp(\nu)) \stackrel{N\to\infty}{\longrightarrow} 0.
	\end{equation}Then \ch the following holds:
	\begin{enumerate}[label=(\roman{*}),ref=(\roman{*})]
		\item\label{i:outlier} If $\theta\in\mathcal O_{\sigma,\nu}$, then $\lim_{N\to\infty}\lambda_1(\sigma N^{-1/2}\bs X_N + \bs A_N)=H_{\sigma,\nu}(\theta)$ almost surely.
		\item\label{i:edge} If $\theta\in\R\setminus\mathcal O_{\sigma,\nu}$, then $\lim_{N\to\infty}\lambda_1(\sigma N^{-1/2}\bs X_N + \bs A_N)=\max\supp\nu$ almost surely.
	\end{enumerate}\ech
\end{lem}
\begin{proof}
We abbreviate $\bs M_N:=\sigma N^{-1/2}\bs X_N + \bs A_N$. Consider an orthogonal diagonalization $\bs A_N = \bs O_N^T \bs D_N \bs O_N$ of $\bs A_N$. As $\bs O_N \bs X_N \bs O_N^T$ is again distributed as a GOE, and as $\bs M_N$ has the same eigenvalues as $\bs O^T \bs M_N \bs O_N$, we henceforth assume w.\,l.\,o.\,g.\ that $\bs A_N$ is diagonal.

We infer the assertions from Theorem~8.1 of~\cite{CDFF} in the case that $\nu$ has compact support.
First, the proof of Theorem 8.1 of~\cite{CDFF} passes through for GOE (in place of GUE) when Theorem~5.1 of~\cite{CDFF} is replaced with Theorem~4.2 of~\cite{Fan}.
We write $\gamma_1:=\theta$. We assume w.\,l.\,o.\,g.\ that $r$ is the minimal integer satisfying assumption~\eqref{e:ass-bulk}.
For any subsequence of $N$ tending to infinity, we find a subsubsequence $(N_i)$ tending to infinity along which $\lambda_j(\bs A_{N_i})$ converges to some $\gamma_j\in(s,\theta]$ for all $j=2,\ldots, r-1$, using compactness of the interval $[s,\theta]$ and minimality of $r$.
Hence, there exists a diagonal matrix $\bs{\wt A}_{N_i}$ with eigenvalues $\lambda_j(\bs{\wt A}_{N_i})=\gamma_j$ whose difference to $\bs A_{N_i}$ vanishes in the Frobenius norm
\begin{equation}\label{e:p-maxgen-F}
	\|\bs{\wt A}_{N_i} - \bs A_{N_i}\|_{\rm F} \overset{i\to\infty}\longrightarrow 0.
\end{equation}
From~\eqref{e:p-maxgen-F}, it follows that $\bs A_{N_i}$ can be replaced with $\bs{\wt A}_{N_i}$ in the definition of $\bs M_{N_i}$ without changing the limiting largest eigenvalue by Lemma~\ref{l:F-max}.

\ch We now assume that $\theta\in\mathcal O_{\sigma,\nu}$ and show assertion~\ref{i:outlier}. In this case, $\bs A_{N_i}$ satisfies the assumptions of Theorem~8.1 1) \ech of~\cite{CDFF},
which yields
\begin{equation}
	\label{e:CDFF-subseq}
	\lim_{i\to\infty}\lambda_{1}(\bs M_{N_i})=H_{\beta,\nu}(\theta)\quad\text{a.\,s.}
\end{equation}
As the limit in~\eqref{e:CDFF-subseq} does not depend on the choice of the subsequence of $N$, it also holds for the original sequence along which $N\to\infty$.

It remains to consider the case of the more general $\nu$ in the assertion. For this, we use truncation arguments for matching upper and lower bounds.

\emph{Lower bound.} For $m\in \R_+$, we consider
\begin{equation}
	V_m:=\left\{ i=1,\ldots,N:\: A_{ii} \ge -m \right\}
\end{equation}
\ch which records the diagonal entries of $\bs A_N$ that have a value at least $-m$. \ech The number of those \ch diagonal entries \ech will be denoted by $N_m=\# V_m$, and we set $r_{m,N}:=\sqrt{N_m/N}$. 
Now,
\begin{multline}\label{e:p-maxgen-lb}
	\lambda_1\left(\bs M_N \right) = \sup \left\{\bs v^T\bs M_N \bs v : \bs v\in \R^N, \|\bs v\|_2 =1 \right\}\\
	\ge   \sup \left\{\bs v^T\bs M_N \bs v : \bs v\in \R^N, \|\bs v\|_2 =1, \max_{i\notin V_m} |v_i|=0 \right\}\\
	= r_{m,N}\sup \left\{\bs v^T\bs M^{(m)}_N \bs v : \bs v\in \R^{N_m}, \|\bs v\|_2 =1 \right\}
	=r_{m,N}\lambda_1\left( \bs M^{(m)}_N\right)
\end{multline}
where
\begin{equation}
	\bs M^{(m)}_N= \sigma N_m^{-1/2}\bs X^{(m)}_N + \bs A^{(m)}_N,\quad
	\bs X^{(m)}_N = (X_{f(i),f(j)})_{i,j\le N_m},\quad
	\bs A^{(m)}_N = {r_{m,N}}^{-1}(A_{f(i),f(j)})_{i,j\le N_m}
\end{equation}
and $f(i)$ denotes the $i$-th largest integer in $V_m$.
Note that $\bs X^{(m)}_N$ is again a GOE of size $N_m$, that $\lim_{N\to\infty}r_{m,N}^2=\nu([-m,s])=:r_m^2$ for all but countably many $m$, and that $\mu_{\bs A^{(m)}_N}$ weakly converges to the probability measure $\nu_m$ as $N\to\infty$, where $\nu_m$ is defined as the image measure of $\nu(\cdot\cap[-m,s])/\nu([-m,s])$ under the dilation $t\mapsto r_{m}^{-1}t$. Also $r_m\to 1$ as $m\to\infty$ by definition of $r_m$.
For $m\ge -2\theta$ and all $N$, we have $r_{m,N}\lambda_1(\bs A^{(m)}_{N})=\theta$, and hence $\lambda_1(\bs A^{(m)}_{N})\to r_m^{-1}\theta$ as $N\to\infty$. Moreover, we note that $\theta\notin\supp\,\nu_m$ for sufficiently large $m$, and from~\eqref{e:def-Hz}, we obtain $\lim_{m\to\infty} H_{\sigma,\nu_m}(\theta)=H_{\sigma,\nu}(\theta)$.
By differentiating~\eqref{e:def-Hz} and using the definition~\eqref{e:def-Stieltjes} of the Stieltjes transform, we also obtain
\begin{equation}
	H_{\beta,\nu_m}'(\theta)=1-\beta^2\int \frac{\nu_m(\rmd x)}{(\theta - x)^2}
	=1-\frac{\beta^2}{\nu[-m,s]}\int_{[-m,s]} \frac{\nu(\rmd x)}{(\theta - r_m^{-1} x)^2},
\end{equation}
which converges to $H'_{\beta,\nu}(\theta)$ as $m\to\infty$. Hence, $H'_{\sigma,\nu_m}(\theta)>0$ for sufficiently large $m$. As $\nu_m$ is compactly supported,  the first part of the proof yields $\lim_{N\to\infty}\lambda_1(\bs M_N^{(m)})= r_m H_{\sigma,\nu_m}(\theta)$ a.\,s. Using~\eqref{e:p-maxgen-lb} and taking $m\to\infty$ yields
$\liminf_{N\to\infty}\lambda_1(\bs M_N)\ge H_{\sigma,\nu}(\theta)$ almost surely.

\emph{Upper bound. } We use the truncation $\bs{\widehat A}^{(m)}_N:=\diag(A_{ii}\vee (-m))_{i=1,\ldots,N}$, and we set $\bs{\wh M}^{(m)}_N:=\sigma N^{-1/2}\bs X_N + \bs{\wh A}^{(m)}_N$.
In place of~\eqref{e:p-maxgen-lb}, we then have
\begin{equation}\label{e:p-maxgen-ub}
	\lambda_1(\bs M_N)=\sup \left\{\bs v^T\bs M_N \bs v : \bs v\in \R^N, \|\bs v\|_2 =1 \right\}\le \lambda_1(\ch\bs{\wh M}_N^{(m)}\ech)
\end{equation}
as
\begin{equation}
	\bs v^T \bs{ A}_N \bs v = \sum_{i=1}^N v_i^2 A_{ii} \le \sum_{i=1}^N v_i^2 \wh A_{ii}^{(m)}=	\bs v^T \bs{\wh A}_N^{(m)} \bs v.
\end{equation}
The empirical spectral distribution $\mu_{\bs{\wh A}^{(m)}_N}$ weakly converges to
\begin{equation}\label{e:def-hat-num}
	\wh \nu_m:= \nu(\cdot\cap[-m,s])+\nu(-\infty,-m)\delta_{-m},
\end{equation}
and we conclude in the same way as for the lower bound.

\bch To show assertion~\ref{i:edge}, we assume that $\theta\in\R\setminus\mathcal O_{\sigma,\nu}\setminus\supp\nu$. Then it follows that $\theta\in\overline{\{u:H'_{\sigma,\nu}(u)<0\}}$ by continuity (see also \cite{CDFF}, p.\,1754). Moreover, assumption~\eqref{e:ass-bulk} implies that $\theta\ge \max\supp\nu$. For $\wh\nu_m$ defined by~\eqref{e:def-hat-num}, we have $H'_{\sigma,\wh\nu_m}(\theta)<H'_{\sigma,\nu}(\theta)\le 0$, and for $m\ge -s$, we have $\theta>\max\supp\wh\nu_m$. If $\bs A_N$ is replaced with the truncated matrix $\bs{\wh A}^{(m)}_N$ from before, assertion~\ref{i:edge} thus follows from Theorem~8.1~2a) of~\cite{CDFF}. For the original diagonal matrix $\bs A_N$, we infer from~\eqref{e:p-maxgen-ub} as before that $\bbP(\lambda_1(\bs M_N) \ge \max\supp\nu + \epsilon)\to 0$ as $N\to\infty$ for each $\epsilon>0$. By definition of $\supp\nu$, we then also have $\bbP(\lambda_1(\bs M_N) \le \max\supp\nu - \epsilon)\to 0$, and assertion~\ref{i:edge} follows.\ech
\end{proof}

\section{Conditional Hessian}
\label{s:Cond}
To analyze the spectral behavior of the Hessian $\bs H^{(k)}$ from~\eqref{e:def-Hk0}, it is useful to condition on the $\sigma$-algebra $\mathcal G_{k-1}$ with respect to which the magnetization $\bs m^{(k)}$ is measurable. Under this conditioning, \bch the distribution of $\bar{\bs g}^{(k)}$ is explicitly known from\ech~\eqref{e:gCov}. In the present section, we show that up to a negligible additive error (as for Lemma~\ref{l:sp}), $\bar{\bs g}^{(k)}$ can be considered as a GOE also under the conditioning on $\mathcal G_{k-1}$. Thus we obtain a representation of $\bs H^{(k)}$ as the sum of a GOE and independent $\mathcal G_{k-1}$-measurable terms.

\bch We recall that the matrix $\bar{\bs g}$ is obtained by $\bar{\bs g}=\tfrac{1}{\sqrt 2}(\bs g + \bs g^T)$. For the covariance matrix $\boldsymbol{V}^{(k)}=(V^{(k)}_{ij,st})_{i,j,s,t\le N}$ of $\bs g$ under $\bbP\left(\cdot\mid \mathcal G_{k-1}\right)$, we give some properties \ech which follow from its definition~\eqref{e:gCov} in terms of the projection $\boldsymbol{Q}^{(k)}$. In the following, we will denote by $\bbP_{k-1}:= \bbP\left(\cdot \mid \mathcal G_{k-1}\right)$ with associated expectation $\E_{k-1}$ the conditional probability given $\mathcal G_{k-1}$.

\begin{lem}
	\label{l:Cov}
	The matrix $\boldsymbol{V}^{(k)}$ is a projection, that is, $\boldsymbol{V}^{(k)}={\boldsymbol{V}^{(k)}}^2$. Furthermore,
	$\boldsymbol{V}^{(k)}=\boldsymbol{1} + \boldsymbol{J}$ for a matrix $\boldsymbol{J}$, where $\bs J$ has eigenvalue $-1$ with multiplicity \ch $N^2-(N-k+1)^2$, \ech and all other eigenvalues are zero.
\end{lem}
\begin{proof}
	By~\eqref{e:def-P} and as $\boldsymbol{Q}^{(k-1)}$ is a projection,
	\begin{equation}
		V^{(k)}_{ij,st}=Q^{(k-1)}_{is}Q^{(k-1)}_{jt}=\bigg(\sum_{u=1}^N Q^{(k-1)}_{iu}Q^{(k-1)}_{us}\bigg)
		\bigg(\sum_{v=1}^N Q^{(k-1)}_{jv}Q^{(k-1)}_{vt}\bigg)
		=\sum_{u,v=1}^N V^{(k)}_{ij,uv}V^{(k)}_{uv,st},
	\end{equation}
	which shows that $\boldsymbol{V}^{(k)}$ is a projection.
	
	To show the assertion on the eigenvalues of $\boldsymbol{J}$, we first note that
	\begin{equation}
		\boldsymbol{Q}^{(k-1)}= \boldsymbol{1}- \bs P^{(k-1)} = \boldsymbol{1}-\frac1N \sum_{s=1}^{k-1} \phi^{(s)} {\phi^{(s)}}^T= \boldsymbol{1}-\boldsymbol{O}\left(\sum_{s=1}^{k-1}\bs D^{(s)}\right) {\boldsymbol{O}^{T}},
	\end{equation}
	where $\boldsymbol{O}$ is an orthogonal matrix and $\bs D^{(s)}$ are diagonal matrices with one \ch entry equal to $1$, and the other entries equal to $0$. \ech The last equality is due to the fact that $\bs P^{(k-1)}$ is a sum of projectors of rank 1 to orthogonal subspaces: thus, these projectors are orthogonally diagonalisable in the same basis. Let
	\begin{equation}
		\bs D = \boldsymbol{1}-\sum_{s=1}^{k-1} \bs D^{(s)},
	\end{equation}
	one readily checks that $\bs D$ has $k-1$ entries that are equal to $0$ and the rest equal to $1$.
	Defining $\bs J$ by $\ch \bs V^{(k)} = \bs 1 + \bs J\ech$ and using the definition~\eqref{e:gCov} of $\ch \boldsymbol{V}^{(k)} \ech$, we obtain
	\begin{equation}\label{e:l:Cov-p1}
		J_{ij,st}=Q^{(k-1)}_{is}Q^{(k-1)}_{jt}-\delta_{ij,st}=
		\sum_{u,v=1}^N O_{ui}D_{uu}O_{us}O_{vj}D_{vv}O_{vt} -\delta_{ij,st}.
	\end{equation}
	Next we define $\boldsymbol{\wt O}$ and $\boldsymbol{\wt D}$ by $\wt O_{ij,st} = O_{is}O_{jt}$ and $\wt D_{ij,st} = D_{is}D_{jt}$.
	Then $\boldsymbol{\wt O}$ is orthogonal as
	\begin{equation}
		(\bs{\wt O}^T\bs{\wt O})_{ij,st}=\sum_{u,v=1}^N O_{i,u} O_{j,v} O_{u,s} O_{v,t} = (\bs O^T \bs O)_{is}(\bs O^T \bs O)_{jt} =\delta_{is}\delta_{jt}.
	\end{equation}
	Hence, we get from~\eqref{e:l:Cov-p1} that $\boldsymbol{J}=\boldsymbol{\wt O}^T (\boldsymbol{ \wt D} - \boldsymbol{1})  \boldsymbol{\wt O}$.
	As the diagonal matrix $(\boldsymbol{ \wt D} - \boldsymbol{1})$ has $\ch(N-k+1)^2\ech$ entries equal to \ch zero\ech, the other entries being \ch $-1$, \ech the assertion follows.
	
\end{proof}
As a consequence, we can approximate $\bar{\bfg}^{(k)}$ by a GOE:
\begin{lem}\label{l:GOE-k}
	Under $\bbP\left(\cdot\mid\mathcal G_{k-1}\right)$, there exists a GOE $\bf X$ such that
	$\bar\bfg^{(k)}=\bs X + \bs Y$ and $\ch N^{-1/2} \ech \|\bs Y\|_{\rm F}$ is tight in $N$.
\end{lem}
\begin{proof}
	From~\eqref{e:gCov} and as $\bs V^{(k)}= {\bs V^{(k)}}^2$ by Lemma~\ref{l:Cov}, there exists a vector $\bs Z=(Z_{ij})_{ij\le N}$ of length $N^2$ whose entries are iid standard Gaussians, such that $g^{(k)}_{ij}=(\bs V^{(k)} \bs Z)_{ij}$ for all $i,j \leq N$. Using again Lemma~\ref{l:Cov}, we diagonalize $\bs V^{(k)} - \bs 1 = \bs O^T \bs D \bs O$, where $\bs O$ is \ch a $\mathcal G_{k-1}$-measurable \ech orthogonal matrix, and $\bs D$ is a deterministic diagonal matrix with \ch $N^2-(N-k+1)^2$ \ech entries equal to $-1$ and the rest to $0$.
	Then we write
	\begin{equation}
		g^{(k)}_{ij} = \left((\bs V^{(k)}-\bs 1) \bs Z\right)_{ij}+Z_{ij}=\left(\bs O^T \bs D\bs O \bs Z\right)_{ij}+Z_{ij}.
	\end{equation}
	We have $\bar\bfg=\bs Y + \bs X$, where
	\begin{equation}
		Y_{ij}:= \frac{1}{\sqrt2}[\left(\bs O^T \bs D\bs O \bs Z\right)_{ij}+ \left(\bs O^T \bs D\bs O \bs Z\right)_{ji}],\quad X_{ij}:= \frac{1}{\sqrt2}[Z_{ij}+Z_{ji}],
	\end{equation}
	where $\bs X,\bs Y$ are $N \times N$ matrices. It remains to prove that $\ch N^{-1/2} \ech \|\bs Y\|_{\rm F}$ is tight in $N$. By a simple convexity argument,
	\begin{equation}
		\|\bs Y\|_{\rm F}^2= \sum_{i,j=1}^{N}Y_{ij}^2\leq \sqrt2 \sum_{i,j=1}^{N}\left[ \left(\left(\bs O^T \bs D\bs O \bs Z\right)_{ij}\right)^2+ \left(\left(\bs O^T \bs D\bs O \bs Z\right)_{ji}\right)^2\right].
	\end{equation}
	By symmetry, it remains to consider
	\begin{equation}
		\sum_{i,j=1}^{N}\left(\left(\bs O^T \bs D\bs O \bs Z\right)_{ij}\right)^2 = \|\bs O^T \bs D\bs O \bs Z \|_{\ell_2(\R^{N^2})}^2
	\end{equation}
	and to show that this expression, \bch when multiplied by $N^{-1}$, \ech is tight in $N$. As the $\ell_2$-norm is invariant under orthogonal transformations, and as $\bs O\bs Z$ is again standard Gaussian distributed, we have
	\begin{equation}
		\| \bs O^T \bs D\bs O \bs Z\|_{\ell_2(\R^{N^2})} = \|\bs D\bs O \bs Z\|_{\ell_2(\R^{N^2})} \stackrel{d}{=} \| \bs D \bs Z\|_{\ell_2(\R^{N^2})}.
	\end{equation}
	Note that \ch $N^2-(N-k+1)^2$ many entries \ech of the vector $\bs D \bs Z$ are $\mathcal  N(0,1)$-distributed, the other entries being $0$. Therefore, \ch by the law of large numbers, $N^{-1}\|\bs D \bs Z\|_{\ell_2(\R^{N^2})}^2$ \ech is tight in $N$, which yields the assertion. 
\end{proof}
We consider the Hessian $\bs H^{(k)}$ from~\eqref{e:def-Hk0}
which reads
\begin{multline}\label{e:def-H}
	H^{(k)}_{ij}=\frac{\beta}{\sqrt{N}} \bar g_{ij} + \frac{2\beta^2}{N} m^{(k)}_i m^{(k)}_j,\quad i,j=1,\ldots,N, i\ne j\\
	H^{(k)}_{ii}=-\beta^2 \left(1-\frac1N\sum_{p=1}^N {m^{(k)}_p}^2\right) -\frac{1}{1-{m^{(k)}_i}^2} +\frac{2\beta^2}{N} {m^{(k)}_i}^2. 
\end{multline}
Now we obtain the following approximation under $\bbP$:
\begin{lem}\label{l:H}
	Let \ch the diagonal matrix \ech $\bs A^{(k)}$ be defined by
	\begin{equation}
		A^{(k)}_{ii} = - \frac{1}{1-{m^{(k)}_i}^2}, i=1,\ldots, N,\quad A^{(k)}_{ij}=0 \text{ for }i\ne j,
	\end{equation}
	and let
	\begin{equation}\label{e:def-tA}
		\bs B^{(k)} = 	2\beta^2 \bfm^{(k)}\otimes \bfm^{(k)}
		+ \beta\sum_{s=1}^{k-1} \big( \zeta^{(s)} \otimes \phi^{(s)} + \phi^{(s)} \otimes \zeta^{(s)}\big).
	\end{equation}
	Then, \bch under $\bbP\left(\cdot\mid\mathcal G_{k-1}\right)$, there is a GOE matrix $\bs X$ such that \ech
	\begin{equation}\label{e:H-tH} 
		\bs H^{(k)}= \frac{\beta}{\sqrt{N}} {\bs X} + \bs A^{(k)} +\bs B^{(k)} -\beta^2(1-q)\bs{1}
		+ \bs R -\epsilon\boldsymbol{1}
	\end{equation}
	where, \bch in probability, \ech $\ch N^{-1} \ech\|\bs R\|_{\rm F}\to 0$ and $\epsilon\to 0$, as $N\to\infty$.
\end{lem}
\begin{proof}
	We set $\epsilon = \beta^2(q-\|\bfm^{(k)}\|^2)$, then $\epsilon\to 0$ in probability as $N\to\infty$ by~\eqref{e:mq}.
	Using the definitions~\eqref{e:def-H}, \eqref{e:def-gk} and~\eqref{e:def-rho}, we can then set
	\begin{equation}\label{e:lp:H}
		\bs R = -\sqrt2 \sum_{s=1}^{k-1} \langle \phi^{(s)},\xi^{(s)} \rangle (\phi^{(s)}\otimes\phi^{(s)})+ \beta N^{-1/2}\bs Y
	\end{equation}
	with $\bs Y$ from Lemma~\ref{l:GOE-k}, so that $N^{-1}\|\bs R\|_F$ converges to zero in probability: for the first term on the r.h.s.\ of~\eqref{e:lp:H}, we note that $\|\phi^{(s)}\otimes\phi^{(s)}\|^2_{\rm F}=\|\phi^{(s)}\|^2=1$. \bch Taking the GOE matrix $\bs X$ from Lemma~\ref{l:GOE-k},  it then \ech suffices to show for each $s$ that
	$ \langle \phi^{(s)},\zeta^{(s)} \rangle \to 0$ $\bbP\left(\cdot\mid\mathcal G_{k-1}\right)$\bch -a.s.\ \ech as $N\to\infty$. This, however, follows from Lemma~11 of~\cite{b2} which states that  $\langle \phi^{(s)},\xi^{(s)} \rangle$ is a centered Gaussian with variance $1/N$ under $\bbP$, hence it converges to $0$ $\bbP$-a.\,s.\ by \bch the Borel-Cantelli Lemma. \ech
\end{proof}

\section{Proof for weak limit of spectral distribution}
\label{s:p-H0}
The proof of Theorem~\ref{t:H0} comes in two parts: first we show, using Bolthausen's algorithm, that $\bs H^{(k)}$ can be considered asymptotically as $N\to\infty$ followed by $k\to\infty$ as the sum of a GOE with variance $\beta/N$, a deterministic diagonal matrix $-\beta^2(1-q)\bs 1$, and an independent diagonal matrix with independent entries having distribution
\begin{equation}\label{e:def-nu}
	\nu:=\mathcal{L}\left(-\frac{1}{1-\tanh^2(h+\beta\sqrt{q}Z)}\right)\,,
\end{equation}
$Z$ being a standard Gaussian. The limiting spectrum of such a sum can be characterized as a free convolution.
We also set $\hat\nu:=\nu(\cdot + \beta^2(1-q))$, then $\hat\nu$ is the image measure of $\nu$ under the shift $t\mapsto t-\beta^2(1-q)$.

\begin{proof}[Proof of Theorem~\ref{t:H0}]
We can rewrite $\bs B^{(k)}$ as
\begin{equation}\label{e:B-proj}
		\bs B^{(k)} = 	2\beta^2 \bfm^{(k)}\otimes \bfm^{(k)}
	+ \tfrac12\beta\sum_{s=1}^{k-1} \left[\big( \zeta^{(s)} + \phi^{(s)}\big) \otimes \big(\zeta^{(s)} + \phi^{(s)} \big)
	-\big( \zeta^{(s)} - \phi^{(s)}\big) \otimes \big(\zeta^{(s)} - \phi^{(s)} \big) \right]
\end{equation}
which is a sum of $2k-1$ matrices of rank $1$. Hence, by Lemma \ref{l:sp} (and induction over $k$),  $\bs B^{(k)}$ has no influence on the limiting spectral distribution of $\bs H^{(k)}$ as $N\to\infty$. Thus, the empirical spectral distribution of $\bs M:=\beta N^{-1/2}\bs X+\bs A^{(k)}-\beta^2(1-q)\bs 1$ converges by Lemmas~\ref{l:sp} and~\ref{l:H} and Slutzky's lemma to the same weak limit as $\mu_{\bs H^{(k)}}$ a.\,s.\ as $N\to\infty$ followed by $k\to\infty$.

\ch By Lemma~\ref{l:LLN}, we have
\begin{equation}
	\int\mu_{\bs A^{(k)}}(\rmd x) f(x)\longrightarrow \int \nu(\rmd x) f(x)
\end{equation}
in probability as $N\to\infty$ for each bounded and continuous $f$. This convergence also holds simultaneously for a countable set of functions $f$ such as the polynomials in $\tanh(x)$ with rational coefficients. By Skorohod coupling, we may assume that this simultaneous convergence holds a.s., so that we can deduce that the weak convergence $\mu_{\bs A^{(k)}}\to \nu$ holds a.s.\ as $N\to\infty$. Using that $\bs X$ and $\bs A^{(k)}$ are independent, we now condition on $\bs A^{(k)}$ and apply Lemma~\ref{l:addconv} to infer that the empirical spectral distribution of $\bs{ M}$ converges a.\,s.\ in the weak topology as $N\to\infty$ to the free additive convolution $\mu_\beta\boxplus \hat\nu$. Without the Skorohod coupling, the empirical spectral distribution of $\bs{ M}$ still converges weakly in distribution to $\mu_\beta\boxplus \hat\nu$. \ech The assertion now follows from Lemma~\ref{l:suppfc} below.
\end{proof}
The support $\supp\,\mu$ of a probability measure $\mu$ on $\R$ is defined by
\begin{equation}
	\supp \mu :=\R\setminus \left\{t\in\R: \exists \epsilon>0 \text{ with } \mu(t-\epsilon, t+\epsilon)=  0\right\}.
\end{equation}
\begin{lem}\label{l:suppfc}
	The free additive convolution $\mu_\beta\boxplus \nu$ has support of the form $(-\infty,d]$ with $d<\beta^2(1-q)$ below and above the AT line (i.e.\ if~\eqref{e:AT} holds with strict inequality or if~\eqref{e:AT} does not hold), and $d=\beta^2(1-q)$ on the AT line (i.e.\ if~\eqref{e:AT} holds with equality).
\end{lem}
\begin{proof}
	 Let $H_{\beta,\nu}(z)$ be defined by~\eqref{e:def-Hz} and $\mathcal O_{\beta,\nu}$ by~\eqref{e:def-O}.
	From the work of Biane~\cite{Biane}, see Proposition~2.2 of~\cite{CDFF}, we have the equivalence
	\begin{equation}
		x\in\R\setminus\supp\mu_\beta\boxplus \nu\quad\Leftrightarrow\quad \exists u\in\mathcal O_{\beta, \nu}\text{ such that } x=H_{\beta, \nu}(u),
	\end{equation}
	noting that the proof of Proposition~2.2 of~\cite{CDFF} passes through even though our $\nu$ is not compactly supported.
Let
\begin{equation}
	d:=\inf_{u\in \mathcal O_{\beta,\nu}} H_{\beta,\nu}(u).
\end{equation}
We note that $\supp \nu=(-\infty,-1]$ and
\begin{equation}\label{e:phi'}
	H_{\beta,\nu}'(u)=1-\beta^2\EE\left(\frac{1}{(u+\frac{1}{1-\tanh^2(h+\beta\sqrt{q} Z)})^2}\right),
\end{equation}
For $u=0$, we evaluate
\begin{equation}\label{e:H0}
	H_{\beta,\nu}(0)=\beta^2\EE\left( 1-\tanh^2(h+\beta\sqrt{q} Z\right)=\beta^2(1-q).
\end{equation}

From~\eqref{e:phi'} and as $1-\tanh^2(x)=\cosh^{-2}(x)$, we can rewrite
\begin{equation}
	H'_{\beta,\nu}(0)=1-\beta^2\EE\cosh^{-4}(h+\beta\sqrt{q}Z).
\end{equation}
Hence, the AT condition~\eqref{e:AT} is equivalent to $H_{\beta,\nu}'(0)\ge 0$, and that~\eqref{e:AT} with strict inequality is equivalent to~$H_{\beta,\nu}'(0)> 0$.  Moreover, \eqref{e:phi'} shows that $H_{\beta,\nu}'(u)$ is strictly increasing in $u\in(-1,\infty)$. From~\eqref{e:def-Hz} and as $\supp\nu=(-\infty,-1]$, we obtain that $H_{\beta,\nu}$ exists and is analytic in $(-1,\infty)$.

We first consider $(\beta,h)$ that satisfy~\eqref{e:AT} with strict inequality. Then from $H_{\beta,\nu}'(0)>0$, we infer that $H_{\beta,\nu}$ attains its infimum over $\mathcal O_{\beta,\nu}$ at some $u_*<0$, and $d=H_{\beta,\nu}(u_*)<H_{\beta,\nu}(0)=\beta^2(1-q)$.

Next, we consider the case hat $(\beta,h)$ does not satisfy~\eqref{e:AT}. Then from $H_{\beta,\nu}'(0)<0$, we infer that $H_{\beta,\nu}$ attains its infimum over $\mathcal O_{\beta,\nu}$ at some $u_*>0$, that $H_{\beta,\nu}$ is decreasing in $(0,u_*)$, and hence $d=H_{\beta,\nu}(u_*)<H_{\beta,\nu}(0)=\beta^2(1-q)$.

For $(\beta,h)$ satisfying~\eqref{e:AT} with equality, we have $H_{\beta,\nu}'(0)=0$, and $H_{\beta,\nu}$ attains its infimum over $\mathcal O_{\beta,\nu}$ at $0$, which implies $d=H_{\beta,\nu}(0)=\beta^2(1-q)$.
\end{proof}

\section{Proof of Theorem~\ref{t:pos}}
\label{s:p-pos-ind}

\ch
\begin{proof}[Proof of Theorem~\ref{t:pos}]
As in~\eqref{e:def-H}, we evaluate the Hessian $\bs H(\bs m)$ from~\eqref{e:def-Hm} in $\bs m\in[-1,1]^N$ as
\begin{equation}
	\bs H(\bs m) = \frac{\beta}{\sqrt{N}}\bs{\bar g} + \bs A_N -\beta^2(1-q)\bs 1 + \eta \bs 1
\end{equation}
where now
\begin{equation}
	A_{N,ij} = \frac{-\delta_{ij}}{1-m_i^2} + \frac{2\beta^2}{N}m_im_j\,,\quad i,j=1,\ldots,N,
\end{equation}
\begin{equation}
	\eta= \frac{\beta^2}{N}\sum_{i=1}^N m_i^2 - \beta^2 q\,.
\end{equation}
The assumptions and~\eqref{e:q} imply $\frac1N\sum_{i=1}^N m_i^2 \to \EE\tanh^2(h+\beta\sqrt{q}Z)=q$ as $N\to\infty$.

First we study the eigenvalues of the matrix $\bs A_N$ via its resolvent.
For $u>0$, we define the diagonal matrix $\bs D =\diag\left(u + \frac{1}{1-m_i^2}\right)$, so that the resolvent reads
\begin{equation}
	(u\bs 1 - \bs A_N)^{-1} = (\bs D - 2\beta^2\bs m\otimes\bs m)^{-1} = \bs D^{-1/2} \left(\bs 1 - 2\beta^2 (\bs D^{-1/2}\bs m)\otimes(\bs m\bs D^{-1/2})\right)^{-1}\bs D^{-1/2}.
\end{equation}
The Sherman-Morrison Lemma~\cite{SM} thus gives that  $u\bs 1 - \bs A_N$ is invertible if and only if $2\beta^2\tr \bs D^{-1}\bs m\otimes\bs m\ne 1$. The latter condition is equivalent to
\begin{equation}\label{e:p-SM}
	\frac{2\beta^2}{N}\sum_{i=1}^N \frac{m_i^2}{u+\frac{1}{1-m_i^2}}\ne 1
\end{equation}
and also to $\bs A_N$ having an eigenvalue at $u$.

To show (ii), we now assume that $\bs m_N\in \bar P^{2,\epsilon}_N$ and let $\epsilon'>0$.
The expression on the left-hand side of~\eqref{e:p-SM} converges to
\begin{equation}\label{e:p-SM-c}
	2\beta^2 \EE \left(\frac{\tanh^2(h+\beta\sqrt{q}Z)}{u+\frac{1}{1-\tanh^2(h+\beta\sqrt{q})}}\right)
\end{equation}
as $N\to\infty$. In particular, for $u=0$ the expression in~\eqref{e:p-SM-c} is larger than $1+\epsilon$ by definition of $\bar P_N^{2,\epsilon}$ and the assumptions. Moreover, \eqref{e:p-SM-c} decreases continuously to $0$ as $u\to\infty$, hence it is equal to $1$ at some $u>0$. It follows that there exist $0<u_{N,-}< u_{N,+}$ and $N_0\in\N$, depending only on $\beta$ and $h$, such that for all $N\ge N_0$, the left-hand side of~\eqref{e:p-SM} is larger than $1+\epsilon'$ when evaluated at $u=u_{N,-}$, and smaller than than $1-\epsilon'$ when evaluated at $u=u_{N,+}$.
As also the expression on the left-hand side of~\eqref{e:p-SM} is continuous, it follows that it is equal to $1$ for some $u_N$ in $(u_{N,-},u_{N,+})$. As $\epsilon'>0$ was arbitrary, it follows that $u_N$ converges to the $u>0$ at which the expression in~\eqref{e:p-SM-c} is equal to $1$.
From~\eqref{e:p-SM}, it then follows that $\bs A_N$ has an eigenvalue at $u_N$. As the expression on the left-hand side of~\eqref{e:p-SM} is decreasing in $u$, it furthermore follows that $\bs A_N$ does not have eigenvalues that are larger than $u_N$. Hence, we have $\lambda_1(\bs A_N)=u_N$.

Let now $u_\infty=\lim_{N\to\infty} u_N$ and $\nu:=\mathcal L(\frac{-1}{1-\tanh^2(h+\beta\sqrt{q}Z)})$. From the proof of Lemma~\ref{l:suppfc}, we recall that $H'_{\beta,\nu}(u)>0$ for all $u>0$. Noting that $u_\infty\in \mathcal O_{\beta,\nu}$, and that assumption~\eqref{e:ass-bulk} is satisfied as $\supp\nu=(\infty,-1]$, we can now apply Lemma~\ref{l:maxgen}\ref{i:outlier} to obtain
	\begin{equation}\label{e:p:1.5}
		\lim_{N\to\infty}\lambda_1\left(\frac{\beta}{\sqrt{N}} \bs{\bar g} + \bs A_{N} \right) =  H_{\beta,\nu}(u_\infty) \quad\text{a.\,s.}
		\end{equation}
	As $H_{\beta,\nu}(u_\infty)-\beta^2(1-q)>0$, assertion (ii) follows.
	
	To show assertion (i), we first note that analogously to the above, there exists $u_\infty>0$ such that the expression in~\eqref{e:p-SM-c} equals $1$ at $u=-u_\infty$, and that this implies $\lambda_1(\bs A_N)\to -u_\infty$ as $N\to\infty$. If $-u_\infty\in\mathcal O_{\beta,\nu}$, the assertion follows as in~\eqref{e:p:1.5} and as $H_{\beta,\nu}(-u_\infty)<H_{\beta,\nu}(0)=\beta^2(1-q)$. If $-u_\infty\notin\mathcal O_{\beta,\nu}$, the assertion follows from Lemma~\ref{l:maxgen}\ref{i:edge} and Theorem~\ref{t:H0}.
\end{proof}

\ech 
\section{Proof of Theorem~\ref{t:pos-macr}}\label{s:p-pos-macr}
In Theorem~\ref{t:pos-macr}, we rely on a specific magnetization $\bfm_N$ at which we evaluate the Hessian of the TAP functional: for $N\in\N$, let $\bs v$ be an eigenvector to the largest eigenvalue of $\beta N^{-1/2}\bar\bfg$ with $\|\bs v\|_2=1$, then we recall that $\beta N^{-1/2} \bs v^T \bar \bfg \bs v \to 2\beta$ a.\,s. For $\alpha\in[0,1]$ to be chosen later, we define the magnetization $\bfm^\alpha_N$ by
\begin{equation}\label{e:def-msign}
	m^\alpha_{N,i} = \alpha\,\sign(v_i),\quad i=1,\ldots,N.
\end{equation}
First we note that for $\beta>0$ and $\alpha^2>1-1/\beta$,
\begin{equation}\label{e:Pl1-alpha}
	\frac{\beta^2}{N}\sum_{i=1}^{N} \left(1-{m^\alpha_{N,i}}^2\right)^2
	=\beta^2\left(1-\alpha^2\right)^2<1,
\end{equation}
and thus $\bfm^\alpha_N\in P^1_N$.
\label{s:p-pos2}
\begin{proof}[Proof of Theorem~\ref{t:pos-macr}]
	Let $\bfm_N=\bfm^\alpha_N$ and $\bs v$ be defined by~\eqref{e:def-msign}. As in~\eqref{e:def-H}, we evaluate $\bs H = \bs H(\bfm_N)$ as follows:
	\begin{equation}
		\begin{split}
			&H_{ij}=\frac{\beta}{\sqrt{N}} \bar g_{ij} + \frac{2\beta^2\alpha^2}{N}\, \sign(v_i)\,\sign(v_j),\quad i,j=1,\ldots,N, i\ne j \\ 
			&H_{ii}=-\beta^2 \left(1- \alpha^2\right) -\frac{1}{1-\alpha^2} +\frac{2\beta^2\alpha^2}{N}. 
		\end{split}
	\end{equation}
	We now estimate $\bs v^T\bs H\bs v$ which is a lower bound for $\lambda_1(\bs H)$.
	First, recall that $\bs v^T\frac{\beta}{\sqrt{N}} \bar\bfg \bs v\to 2\beta$ a.\,s.\ as $N\to\infty$. The random vector $\bs v$ is distributed as the first column of a Haar distributed random matrix on the orthogonal group on $\R^N$ (see e.\,g.\ Corollary~2.5.4 in~\cite{AGZ}). Hence, by Lemma~\ref{l:Haar-abs} below,
	\begin{equation}\label{e:claim-absv}
		\frac{2\beta^2\alpha^2}{N}\sum_{i,j=1}^N v_i\,\sign(v_i)\,\sign(v_j)\, v_j \to \frac{4\beta^2\alpha^2}{\pi}
	\end{equation}
	in probability as $N\to\infty$. It follows that
	\begin{equation}\label{e:pos-macr-p}
		\bs v^T \bs H\bs v\to 2\beta -\beta^2 \left(1- \alpha^2\right) -\frac{1}{1-\alpha^2} +\frac{4\beta^2\alpha^2}{\pi}
	\end{equation}
	in probability as $N\to\infty$. For fixed $\beta>0$, the expression on the r.h.s.\ attains its maximum at
	$\alpha^2=1-\beta^{-1}(1+4/\pi)^{-1/2}$ which is larger than $1-\beta^{-1}$ and hence $\bfm^\alpha\in P^1_N$ by~\eqref{e:Pl1-alpha}.
	The value of the maximum of the r.h.s.\ of~\eqref{e:pos-macr-p} is strictly positive for $\beta>\tfrac{\pi}{2}\left(\sqrt{1+4/\pi}-1\right)=:\beta_0\approx 0.798$.
\end{proof}
\begin{lem}\label{l:Haar-abs}
	Let $\bs v$ be distributed as the first column of a Haar distributed random matrix on the orthogonal group on $\R^N$. Then,
	\begin{equation}\label{e:Haar-abs}
		\frac{1}{\sqrt{N}}\sum_{i=1}^N \left|v_i\right|\to \sqrt{2/\pi}
	\end{equation}
	in probability as $N\to\infty$.
\end{lem}
\begin{proof}
	First we consider the expectation
	\begin{equation}
		\frac{1}{\sqrt{N}}\sum_{i=1}^N\E \left|v_i\right|= \E\sqrt{N}\left|v_1\right|,
	\end{equation}
	which converges to $\E|Z|=\sqrt{2/\pi}$ by \bch e.g. Proposition~2.5 of~\cite{Me} (or by the convergence in the total variation distance from~(1) in~\cite{DF87} along with uniform integrability, obtained from boundedness of the second moment $\E N|v_1|^2=\E\sum_{i=1}^N |v_i|^2 =1$). \ech
	Likewise, for the second moment, we have
	\begin{equation}
		\frac{1}{N}\sum_{i,j=1}^N\E \left|v_i\right|\left|v_j\right|= (N-1)\E \left|v_1\right| \left|v_2\right| + \E v_1^2.
	\end{equation}
	Here the second term on the r.h.s.\ converges to zero, and the first term to $\left(\E|Z|\right)^2$ again
	\bch by e.g.\ Proposition~2.5 of \cite{Me}, or by (1) in~\cite{DF87} along with boundedness of the second moment
	\begin{equation}
		(N-1)^2 \E \left|v_1\right|^2 \left|v_2\right|^2 \le \E\left( \sum_{i=1}^N\left|v_i\right|^2 \sum_{j=1}^N\E\left|v_j\right|^2\right)= 1.
	\end{equation}\ech
	This shows that the variance of the expression on the l.h.s.\ of~\eqref{e:Haar-abs} converges to zero, so that the convergence of the expectation implies the assertion.
\end{proof}

The authors have no competing interests to
declare.

\end{document}